\documentclass[11pt,reqno]{amsart}
\usepackage{lmodern}
\usepackage{lmodern}
\usepackage[utf8]{inputenc}
\usepackage[a4paper]{geometry}
\usepackage{amstext}
\usepackage{amsthm}
\usepackage{amssymb}
\usepackage[unicode=true,
 bookmarks=false,
 breaklinks=false,pdfborder={0 0 1},backref=section,colorlinks=false]
 {hyperref}

\usepackage{comment}

\makeatletter
\numberwithin{equation}{section}
\numberwithin{figure}{section}

\@ifundefined{date}{}{\date{}}
\usepackage{amsthm}
\usepackage{mathrsfs}\usepackage{amsfonts}

\newcommand{\Hm}[1]{\leavevmode{\marginpar{\tiny%
$\hbox to 0mm{\hspace*{-0.5mm}$\leftarrow$\hss}%
\vcenter{\vrule depth 0.1mm height 0.1mm width \the\marginparwidth}%
\hbox to 0mm{\hss$\rightarrow$\hspace*{-0.5mm}}$\\\relax\raggedright
#1}}}

\usepackage{xcolor}
\newcommand{\blue}[1]{\textcolor{blue}{#1}}
\newcommand{\red}[1]{\textcolor{red}{#1}}
\newcommand{\adp}[1]{\textcolor{purple}{#1}}

\allowdisplaybreaks

\numberwithin{equation}{section}

\newtheorem{theo}{Theorem}[section]
\newtheorem{lem}[theo]{Lemma}\newtheorem{prop}[theo]{Proposition}\newtheorem{coro}[theo]{Corollary}\newtheorem{defi}[theo]{Definition}\newtheorem{remark}[theo]{Remark}

\def\E{{\mathbb{E}\ \!\!}}

\def\R{{\mathbb{R}\ \!\!}}
\def\L{{\mathcal{L}\ \!\!}}

\def\RR{{\mathbb{R}\ \!\!}}

\def\P{{\mathbb{P}\ \!\!}}

\def\one{\mathbf 1}
\def\bone{{\mathbf{1}}}

\newcommand{\EE}{\mathbb{E}}

\newcommand{\ve}{\varepsilon}

\def\Var{{\mathrm{{\rm Var}}}}
\def\Hess{{\mathrm{{\rm Hess }}}}

\def\Cov{{\mathrm{{\rm Cov}}}}
\def\cov{{\mathrm{{\rm Cov}}}}
\def\Var{{\mathrm{{\rm Var}}}}
\def\var{{\mathrm{{\rm Var}}}}

\def\and{{\mathrm{{\rm and}}}}

\def\ve{{\varepsilon\ \!\!}}
\def\eps{{\varepsilon\ \!\!}}

\title
{Covariance inequalities for convex and log-concave functions}

\author{Michel~Bonnefont} \address[M.~Bonnefont]{UMR CNRS 5251, Institut de Math\'ematiques de Bordeaux, Universit\'e Bordeaux, France}
\thanks{MB is partially supported by the French ANR-18-CE40-0012 RAGE project and the French ANR-19-CE40-0010-01 QuAMProcs project}
\email{\url{mailto:michel.bonnefont(at)math.u-bordeaux.fr}} \urladdr{\url{http://www.math.u-bordeaux.fr/~mibonnef/}}

\author{Erwan~Hillion} \address[E.~Hillion]{Aix Marseille Univ, CNRS, I2M, Marseille, France}
\email{\url{mailto:erwan.hillion(at)univ-amu.fr}}

\author{Adrien~Saumard} \address[A.~Saumard]{Univ. Rennes, ENSAI, CNRS, CREST - UMR 9194, F-35000 Rennes, France}
\email{\url{mailto:adrien.saumard@ensai.fr}}

\makeatother

\begin{document}
\maketitle

\begin{abstract}
Extending results of Hargé and Hu for the Gaussian measure, we prove inequalities for the covariance $\cov_\mu(f,g)$ where $\mu$ is a general product probability measure on $\R^d$ and $f,g : \R^d \to \R$ satisfy some convexity or  log-concavity assumptions, with possibly some symmetries.
\end{abstract}

\maketitle


\section{Introduction}
If $\mu$ is a probability measure on $\R^d$ and if $f,g \in L^2(d\mu)$ are two square integrable functions with respect to $\mu$, their covariance is defined by 
\begin{align*}
\Cov_{\mu}(f,g)&=   \int \left(f-\int f d\mu \right)\left(g-\int g d\mu\right) d \mu 
\end{align*}
and is a measure of the joint variability of the two functions.
Here and in all the sequel, we make the assumptions that $f$ and
$g$ have enough integrability and regularity conditions, so that
all the written quantities are well defined.

Lying at the intersection of probability, analysis and geometry, covariance identities and inequalities provide a  variety of tools. Without trying to be exhaustive, let us cite some of them: FKG inequalities (\cite{FKG}), (asymmetric) Brascamp-Lieb inequalities (\cite{otto-menz,carlen-cordero-lieb,abj}), Stein kernels (\cite{chatterjee:stein,nourdin-viens,ledoux-nourdin-peccati,Courtadeetal:19,Fathi:stein,saumard:wpi}), concentration inequalities (\cite{bobkov-gotze-houdre,HouPriv:02}, \cite[Section 5.5]{MR1849347}).

The proof techniques of these covariance identities and inequalities vary from semi-group techniques, other types of integration by parts, measure transportation or stochastic calculus. Gaussian measures offer a particularly fruitful framework in this perspective
and in  Theorem \ref{thm:harge} below, we recall famous covariance inequalities known for the standard  Gaussian measure.
The main point of this work is to discuss and extend partially these results  beyond the Gaussian assumption, to the  case of  general product measures.

\begin{theo}\label{thm:harge}Let $\gamma$
be the standard Gaussian distribution on $\R^{d}$.

\begin{enumerate}
\item \cite{hu-chaos,harge:04}  Let $f$ and $g$
be two convex functions on $\R^{d}$, then 
\begin{equation}
\cov_\gamma(f,g)  \geq \cov_\gamma (f,x)\cdot \cov_\gamma (g,x)\label{eq-harge-gaussienne}
\end{equation}
where $\cdot$ denotes the standard scalar product on $\R^{d}$.
\item \cite{harge:04} Let $f$ be a log-concave function and $g$ be a convex function. Assume moreover that $f$ is orthogonal to the linear functions -- that is $\cov_\gamma (f,x)=0$ --, then 
\begin{equation}
\cov_\gamma(f,g)  \leq 0 \label{eq-harge-gaussienne-niv2}.
\end{equation}
\item \cite{Royen} Let $f$  and  $g$ be some  quasi-concave functions that are both even, then
\begin{equation}
\cov_\gamma(f,g)  \geq 0 \label{eq-harge-gaussienne-niv3}.
\end{equation}

\end{enumerate} \end{theo}

The first point of Theorem \ref{thm:harge} is due to Hu \cite{hu-chaos} and was recovered by Hargé
\cite{harge:04}. Hu's proof is based on some Itô-Wiener chaos decomposition. This decomposition is based on the interpolation of the covariance  by the standard heat semi-group.
Hargé's proof of the second point is based on optimal transport theory and Caffarelli's contraction theorem. Hargé obtained in fact an inequality when $f$ is not necessarily orthogonal to the linear functions, which  by a limiting argument recovers (1).

Point (3) was proven by Royen \cite{Royen}. It is known as the Gaussian correlation inequality  and was an open question during decades. We refer to \cite{latala-matlak:royen} and \cite{barthe:royen} for history of this result. Royen proved his result in its geometric form, for symmetric convex bodies, by  approximation with  finite intersections of  symmetric slabs.  The main ingredients are then an interpolation of some  dependent and independent Gaussian measures through their covariance matrix and clever computations of the Laplace transform of  multivariate Gamma distributions.  Royen thus proves its result  for some family of multivariate Gamma distributions.  Subsequently, Eskenazis, Nayar and Tkocz \cite{Eskenazis} noticed that Theorem \ref{thm:harge}(3) still holds for product measures whose marginals are mixtures of centered Gaussian measures. In the Appendix, we also show that Theorem \ref{thm:harge}(2) is true in the latter situation.





\

The first main new results of this paper are devoted to dimension one. In dimension one,  the covariance inequalities of Theorem \ref{thm:harge} are not limited to the Gaussian context but actually hold for \emph{any} probability
measure on $\R$ having a finite variance.

\begin{theo}\label{thm:main-1d} Let $\mu$ be any probability measure
on $\R$ admitting a second moment. 

\begin{enumerate}
\item For any convex functions $f$
and $g$, one has 
\begin{equation}
\Var(\mu)\, \Cov_{\mu}(f,g)\geq\Cov_{\mu}(f,x)\,\Cov_{\mu}(g,x).\label{eq-main-1d}
\end{equation}

\item Let $f$ be a log-concave function and $g$ be a convex function. Assume moreover that $f$ is orthogonal to the linear function $x$, then 
\begin{equation}
\cov_\mu(f,g)  \leq 0 \label{eq-main-1d-niv2}.
\end{equation}
\item  Let $f$  and  $g$ be some  quasi-concave functions that are both even, then
\begin{equation}
\cov_\mu(f,g)  \geq 0 \label{eq-main-1d-niv3}.
\end{equation}
\end{enumerate}
\end{theo}





The fact that Theorem \ref{thm:main-1d} holds for any probability measure whereas Theorem \ref{thm:harge} seems limited to the Gaussian setting is striking and rises the following natural question: what about general product measures? Before trying to answer this question, we shall introduce some notations and the hypotheses.

\

\textbf{Notations and hypotheses:} In all the sequel of the paper, we  consider $\mu= \mu_1 \otimes  \dots \otimes \mu_d$  to be  a product measure 
on $\R^d$. Moreover for  each $1 \leq k \leq d$, we denote  by $a_k(x_k)$  a positive function on $\R$ and by $A_k$ its  primitive, centered with respect to $\mu_k$ and we assume that $A_k \in L^2(\mu_k)$.
When we apply the results with $a_k\equiv 1$, we thus implicitly assume that the measure $\mu_k$ admits a second moment.
We  assume that $f,g\in L^2(\mu)$. In order to apply the tensorization method and to exchange derivative and integral, we also assume that all the first and second  partial   derivatives  of $f$ and $g$ are integrable with respect to $\mu$.
\begin{remark}
It is actually possible to  weaken the regularity assumptions on the second order partial derivatives. It is indeed sufficient to assume that  the first derivatives are monotonic on $\R^d$, at the price of standard approximation arguments. This is particularly transparent with the pure determinantal approach. 
But for clarity and simplicity, we prove the theorems by using the second order partial  derivatives.
\end{remark}

Arguably, the first basic idea to investigate general product measures is to use a tensorization  argument. This allows us to  obtain the following  extension of Theorem \ref{thm:main-1d}(1) to the higher dimensional case.

\begin{theo}\label{thm:Hu-produit-A}
Let $\mu$ be a product measure on $\R^d$.

Let $f$ and $g$ be two functions on $\R^d$ such that   for each pair $ 1\leq i,j \leq d$, 
the signs of 
\begin{equation}\label{eq:cond-l2-fg-mod}
\partial_{j} \left(\frac{\partial_{i}f(x)}{a_{i}(x_{i})}\right)  \textrm{ and  } \partial_{j} \left(\frac{\partial_{i}g(x)}{a_{i}(x_{i})}\right)
\end{equation} 
are constant on $\R^d$ and equal.
Then 
\[
\cov_\mu(f,g) \geq \sum_{i=1}^{d}\frac{1}{\var_{\mu_i}(A_i)}\;\Cov_{\mu}(f(x),A_{i}(x_i))\,\Cov_{\mu}(g(x),A_{i}(x_i)).
\]

\end{theo}
Taking the functions $a_i\equiv 1$ gives the following corollary.
\begin{coro}\label{cor:hu-produit}
Let $\mu$ be a product measure on $\R^d$.
 Let $f$ and $g$ be two functions on $\R^d$ such that   for each couple $ 1\leq i,j \leq d$, the signs of  
\begin{equation}\label{eq:cond-l2-fg}
\partial_{i,j} f(x)  \textrm{ and  } \partial_{i,j}g(x) 
\end{equation} 
are constant and equal.
Then 
\[
\cov_\mu(f,g) \geq \sum_{i=1}^{d}\frac{1}{\Var(\mu_i)}\;\Cov_{\mu}(f(x),x_{i})\,\Cov_{\mu}(g(x),x_{i}).
\]
\end{coro}

A striking point is that Corollary \ref{cor:hu-produit} is not limited to the Gaussian setting, but holds for any product measure with marginals having a finite second moment. Particularizing to the Gaussian case, where $\mu=\gamma$, the conclusion of Corollary \ref{cor:hu-produit} is the same as in Theorem
\ref{thm:harge}(1), but under different assumptions on the functions $f$ and $g$. 
Even if they coincide in dimension one, the two assumptions are different in higher dimensions and are not included one into another. The assumption of Corollary \ref{cor:hu-produit} seems less classical from a geometric point of view than the classical convexity assumption. Actually, such assumption on the sign of the second partial derivatives also appears in the context of Gaussian comparison theorems, see for instance \cite[Theorem 3.11]{LedTal:11}, that implies Slepian's lemma and Gordon's min-max theorem.
Note finally that in the Gaussian setting, even if the statement of Corollary \ref{cor:hu-produit}
seems to be new, its proof could be also deduced from the arguments developed in the proof of Hu \cite{hu-chaos}. 


\begin{remark} The conditions stated in \eqref{eq:cond-l2-fg-mod} can also be written as the conjunction of Conditions \eqref{eq:cond-a-i} and \eqref{eq:cond-a-ij} below: for each $1\leq i \leq d$, the signs of 
\begin{equation}\label{eq:cond-a-i}
\partial_{i} \left(\frac{\partial_{i}f(x)}{a_{i}(x_{i})}\right)  \textrm{ and  } \partial_{i} \left(\frac{\partial_{i}g(x)}{a_{i}(x_{i})}\right)
\end{equation}
are constant and equal 
and for each couple $1\leq i\neq j \leq d $, the signs of 
\begin{equation}\label{eq:cond-a-ij}
\partial_{ij} f(x)  \textrm{ and  }\partial_{ij} g(x)
\end{equation} are constant and equal.
The condition described by Equation~\eqref{eq:cond-a-i} can be interpreted as follows: let $B_i : \R^d \to \R^d$ be the inverse bijection of 
\[ (x_1,\ldots,x_d) \mapsto (x_1,\ldots,A_i(x_i),\ldots,x_d),\] then Condition \eqref{eq:cond-a-i} means that the functions \[x_i \mapsto (f\circ B_i)(x_1,\ldots,x_d) \ , \ x_i \mapsto (g \circ B_i)(x_1,\ldots,x_d)\] are both convex or both concave. 
In the case  where $a_i \equiv 1$ for all $ i=1,\dots, d$, and if moreover,  all the signs in \eqref{eq:cond-a-i} are the same, then the functions $f$ and $g$ are both coordinatewise convex or both coordinatewise concave. 
\end{remark}

We now want to investigate what happens when the functions are assumed to satisfy some symmetries. As we shall see, the good notion that fits with the tensorization argument is quite strong and is the  unconditionality of (at least) one function.

 We  recall that a function $f:\R^d \to \R$ is said to be \emph{unconditional} if it is symmetric with respect to each  hyperplan of coordinates : for all $(x_1,\dots, x_d)\in \R^d$,
\[
f(x_1,\dots, x_d) = f (\eps_1 x_1,\dots, \eps_d x_d).
\]
holds for each choice of signs $( \eps_1,\dots, \eps_d ) \in \{-1,1\}^d$.


Theorem \ref{thm:Hu-tens-incond-A}  below is a multi-dimensional extension of Theorem \ref{thm:main-1d}(1) with a symmetry assumption.

\begin{theo}\label{thm:Hu-tens-incond-A}
Let $\mu= \mu_1 \otimes  \dots \otimes \mu_d$ be a product measure on $\R^d$ and assume  that for each $1\leq i \leq d$, the measure $\mu_i$ is even. 
Let $f$ and $g$ be two functions on $\R^d$ such that   for each $ 1\leq i \leq d$ and  all $x\in \R^d$, the signs of 
\begin{equation}\label{eq:cond-l2-idem}
\partial_{i}\left( \frac{\partial _i f(x)}{a_i(x_i)} \right) \textrm{ and  \ } \partial_{i}\left( \frac{\partial _i g(x)}{a_i(x_i)} \right) 
\end{equation}
are constant and equal.
Assume moreover that one of the functions is unconditional. Then 
\[
\cov_\mu(f,g) \geq 0.
\]
\end{theo}
The following corollary is directly obtained by setting again the functions $a_i$ to be identically equal to $1$.

\begin{coro}\label{cor:Hu-tens-incond}
Let $\mu= \mu_1 \otimes  \dots \otimes \mu_d$ be a product measure on $\R^d$ and assume  that for $1\leq i \leq d$, the measures $\mu_i$ are even. 
Let $f$ and $g$ be two functions on $\R^d$ such that   for each $ 1\leq i \leq d$, the signs of 
\begin{equation}\label{eq:cond-l2-cor}
\partial_{i,i}f(x) \textrm{ and  \ }\partial_{i,i}g(x) 
\end{equation}
are constant and equal.
Assume moreover that one of the functions is unconditional. Then 
\[
\cov_\mu(f,g) \geq 0.
\]
\end{coro}

With these symmetries, the tensorization method also leads to the following extension of Theorem \ref{thm:main-1d}(2) and (3). 

\

\begin{theo}\label{thm:Harge-Royen-tens}
Let $\mu= \mu_1 \otimes  \dots \otimes \mu_d$ be a product measure on $\R^d$. 
\begin{enumerate}
\item Assume that for $1\leq i\leq d$, the marginals $\mu_i$ are even and log-concave. 
Let $f=e^{-\phi} $ be  an unconditional positive log-concave function
and $g$ be a coordinatewise convex function on $\R^d$,
then 
\[
\cov_\mu(f,g) \leq 0.
\]
\item Assume that $f$ and $g$ are both unconditional and  coordinatewise quasi-concave. Then
\[
\cov_\mu(f,g) \geq 0.
\]
\end{enumerate}
\end{theo}


\


A drawback of this tensorization approach is arguably that in Theorem \ref{thm:Hu-tens-incond-A} 
and Corollary \ref{cor:Hu-tens-incond}, we assume a strong symmetry property: the unconditionality of at least one function.  

In order to require less  symmetry assumptions, it is  natural to try to use, instead of the tensorization argument,  a  more global approach. A first attempt would be to use the interpolation with the associated Markov semi-group and the covariance representation  given in \eqref{eq:cov-rep-interpolation}. Actually, we shall provide a slightly different covariance representation, based on an argument of ``duplication'' of random variables (Lemma \ref{lem:var-prod}). 
The main reason for this choice is that the latter approach is much simpler than the semi-group approach and is also more effective. 
See more comments in  Section \ref{sec:comments}.
 As expected, this approach allows us to reduce drastically the symmetries required on the functions, but at prize of considering some convexity type assumptions that are less common.
Theorem \ref{thm:Harge-Royen-global} below provides an extension of Theorem \ref{thm:main-1d}(2) and (3).

\begin{theo}\label{thm:Harge-Royen-global}
Let $\mu$ be a product measure on $\R^d$.
\begin{enumerate}
\item Let $f=e^{-\phi}$ and $g$ be two functions on $\R^d$ such that  
all $x\in \R^d$, 
\begin{equation}\label{eq:cond-ii-phi-g}
\partial_{i} \left( \frac{\partial_i \phi(x)}{a_i(x_i)} \right) \leq 0   \textrm{ and  } \partial_{i} \left( \frac{\partial_i g(x)}{a_i(x_i)} \right)\geq 0 \textrm{ for all } 1\leq i \leq d, 
\end{equation}
and 
\begin{equation}\label{eq:cond-ij-phi-g}
\partial_{i,j} \phi(x) \leq 0  \textrm{ and  } \partial_{i,j }g(x)\geq 0 \textrm{ for all } 1\leq i\neq j \leq d. 
\end{equation}
Assume moreover that $f$ is orthogonal to the functions $A_i(x_i)$ for all $1\leq i\leq d$,
then 
\[
\cov_\mu(f,g)  \geq 0.
\]
\item Let $f=e^{-\phi}$ and $g=e^{-\psi}$ be two functions on $\R^d$ such that  for   all $x\in \R^d$, 
\begin{equation}\label{eq:cond-ii-phi-psi}
\partial_{i} \left( \frac{\partial_i \phi(x)}{a_i(x_i)} \right) \leq 0   \textrm{ and  } \partial_{i} \left( \frac{\partial_i \psi(x)}{a_i(x_i)} \right)\leq 0  \textrm{ for all } 1\leq i \leq d, 
\end{equation}
and 
\begin{equation}\label{eq:cond-ij-phi-psi}
\partial_{i,j} \phi(x) \leq 0  \textrm{ and  } \partial_{i,j }\psi(x)\leq 0 \textrm{ for all } 1\leq i\neq j \leq d. 
\end{equation}
Assume moreover that  the product measure $\mu$ is symmetric and the functions $a_i$ are even for $1\leq i \leq d$ and that also both  $f$ and $g$ are even, 
then 
\[
\cov_\mu(f,g) \geq 0.
\]
\end{enumerate}
\end{theo}

As before, the choice $a_k\equiv1$ is worth looking at and gives the following corollary.
\begin{coro}\label{cor:Harge-Royen-global}
Let $\mu$ be a product measure on $\R^d$.
\begin{enumerate}
\item Let $f=e^{-\phi}$ and $g$ be two functions on $\R^d$ such that   for all $ 1\leq i,j \leq d$ and  all $x\in \R^d$, 
\begin{equation}\label{eq:cond-l2-phi-g}
\partial_{i,j} \phi(x) \leq 0 \textrm{ and  } \partial_{i,j }g(x)\geq 0. 
\end{equation}
Assume moreover that $f$ is orthogonal to the coordinate functions $x_i$ for all $1\leq i\leq d$,
then 
\[
\cov_\mu(f,g)  \geq 0.
\]
\item Let $f=e^{-\phi}$ and $g=e^{-\psi}$ be two functions on $\R^d$ such that   for all $ 1\leq i,j \leq d$ and  all $x\in \R^d$, 
\begin{equation}\label{eq:cond-l2-phi-psi}
\partial_{i,j} \phi(x) \leq 0 \textrm{ and  } \partial_{i,j }\psi (x)\leq 0. 
\end{equation}
Assume moreover that  the product measure $\mu$ is symmetric and that both  $f$ and $g$ are even, 
then 
\[
\cov_\mu(f,g) \geq 0.
\]
\end{enumerate}
\end{coro}


\

\textbf{Outline. \ }The paper is organized as follows. The case of the dimension one is investigated in Sections \ref{sec:Andreev-dim1}, \ref{sec:kernel-dim1} and \ref{sec:cheb}.  In Sections \ref{sec:Andreev-dim1} and  \ref{sec:kernel-dim1}, we produce two different  proofs of Theorem \ref{thm:main-1d}. The first one, given in Section \ref{sec:Andreev-dim1}, is based on the use of determinants and the so-called Andreev's formula. The second one, detailed in Section \ref{sec:kernel-dim1}, is based on a covariance identity due to Hoeffding and the use on $\RR^2$ of the classical FKG inequality. In Section \ref{sec:cheb}, we  notice that more structure is actually present in dimension one: the kernel $k$  in Hoeffding's covariance identity is indeed \emph{totally positive} in the sense of Karlin \cite{karlin:book}. Consequently, determinantal covariance inequalities for general Chebyshev systems follow (see Theorem \ref{thm:cheb} for the precise statement). The latter inequalities are also recovered without  using Hoeffding's covariance identity, through a direct approach with determinants and Andreev's formula.

The tensorization method and the proofs of Theorems \ref{thm:Hu-produit-A}, \ref{thm:Hu-tens-incond-A} and \ref{thm:Harge-Royen-tens} are given in Section \ref{sec:tens}, except for the proofs of Theorem \ref{thm:main-1d}(3) and  Theorem \ref{thm:Harge-Royen-tens}(2), that pertain to the hypothesis of quasi-concavity and are detailed in Section \ref{sec:quasi-concave}. Indeed, the method for proving Theorem \ref{thm:main-1d}(3) in dimension one is very specific and independent from the rest of the paper. Theorem  \ref{thm:Harge-Royen-tens}(2) is then  obtained by tensorization.

In addition, a generalization of Hoeffding's covariance identity for product measures, obtained by a duplication argument, is provided in Section \ref{sec:global}. A second proof of Theorem \ref{thm:Hu-produit-A} and the proof of Theorem  \ref{thm:Harge-Royen-global} are then given. 
As explained above, another natural generalization of Hoeffding's covariance identity would be given through the standard semi-group interpolation. Comments on the difficulty of using this covariance representation are provided in Section \ref{sec:comments}. Some possible examples of applications are given in Section \ref{sec:examples}. Finally, in the Appendix, we also prove that that Theorem \ref{thm:harge}(2) is true  for product measures whose marginals are mixtures of centered Gaussian measures. 


\section{A determinantal approach in dimension one} \label{sec:Andreev-dim1}
 This section is devoted to a first proof of Theorem~\ref{thm:main-1d}(1) and (2). The proof is  based on properties of determinants, in particular on the intertwining between determinants and the integral operator, a property known as Andreev's formula. Similar arguments will be used in Section~\ref{sec:cheb} in the more general framework of Chebyshev systems.

\subsection{Convexity and determinants}

\begin{defi}
A pair of real-valued functions $(u,U)$ is said to satisfy Assumption ($\mathcal{C}$) if for any triple $(x_1,x_2,x_3) \in \RR^3$ with $ x_1 < x_2 < x_3$, one has
\begin{equation*}
D(x_1,x_2,x_3) = \left|\begin{array}{ccc} 1 & 1 & 1 \\ u(x_1) & u(x_2) & u(x_3) \\ U(x_1) & U(x_2) & U(x_3) \end{array} \right| \ge 0.
\end{equation*}
\end{defi}

In other terms, the couple $(u,U)$ satisfies Assumption ($\mathcal{C}$) if and only if the triple $(1,u,U)$ forms a Chebyshev system (see Definition~\ref{def:Cheby}). From elementary properties of determinants, it follows that:

\begin{prop} \label{prop:AssCsignD}
Let $(u,U)$ be satisfying Assumption ($\mathcal{C}$) and $D : \RR^3 \rightarrow \RR$ be as defined above. Let $(x_1,x_2,x_3) \in \RR^3$. Let $\sigma \in S_3$ be a permutation of $\{1,2,3\}$ such that $x_{\sigma(1)} \le x_{\sigma(2)} \le x_{\sigma(3)}$. Then
\begin{equation*}
 \ve(\sigma) D(x_1,x_2,x_3) \ge 0
\end{equation*}
\end{prop}

\begin{prop} \label{prop:ConvC}
A function $f : \RR \to \RR$ is convex if and only if, for any $x\in \RR$, the pair $(x,f(x))$ satisfies ($\mathcal{C}$).
\end{prop}

\begin{proof} Take $(x_1,x_2,x_3) \in \RR^3$ with $x_1 < x_2 < x_3$. Expanding the determinant $D(x_1,x_2,x_3)$ gives:
\begin{equation*}
D(x_1,x_2,x_3) = (x_2-x_1)(f(x_3)-f(x_2)) - (x_3-x_2)(f(x_2)-f(x_1)).
\end{equation*}
Dividing by the positive quantity $(x_3-x_2)(x_2-x_1)>0$, we obtain that $D(x_1,x_2,x_3) \ge 0$ if and only if
\begin{equation*}
\frac{f(x_3)-f(x_2)}{x_3-x_2} \ge \frac{f(x_2)-f(x_1)}{x_2-x_1},
\end{equation*}
which is the slope inequality equivalent to convexity of $f$.
\end{proof}

\medskip

\begin{coro}
If $U$ is an increasing bijection between $\R$ and some interval $I$, then $(1,u,U)$ satisfies ($\mathcal{C}$) if and only if $u \circ U^{-1}$ is concave on $I$.
\end{coro}

\begin{proof}
We notice that, for $x_1 < x_2 < x_3$ we have $$D(x_1,x_2,x_3) = D(U^{-1}(y_1),U^{-1}(y_2),U^{-1}(y_3))$$ for some triple $y_1<y_2<y_3 \in I$. Elementary properties of determinants then give:
\begin{equation*}
D(x_1,x_2,x_3) =  \left|\begin{array}{ccc} 1 & 1 & 1 \\ y_1 & y_2 & y_3 \\ (-u \circ U^{-1})(y_1) & (-u \circ U^{-1})(y_2) & (-u \circ U^{-1})(y_3) \end{array} \right|
\end{equation*}
and we conclude by applying Proposition~\ref{prop:ConvC}.
\end{proof}

\

Let us now consider some positive function $f$. Let $F$ be a primitive of $f$. It is known (see \cite{bobkov:extremal}) that $f$ is log-concave if and only if $f \circ F^{-1}$ is concave. From the above, we deduce the following proposition.

\begin{prop} \label{prop:logconcC}
A positive function $f$ is log-concave if and only if the pair $(f,F)$ satisfies ($\mathcal{C}$).
\end{prop}

\subsection{An Andreev-type formula}
 A key point in this approach is the following Andreev-type formula which exchanges expectation and determinants. 
\begin{prop} \label{prop:Andreev}
Let $(f_i)_{1 \le i \le n}$ and $(g_i)_{1 \le i \le n}$ be two $n$-uples of functions in $L^2(\mu)$. We have:
$$
\det\left(\EE_\mu \left[f_i(X) g_j(X) \right]\right) = \frac{1}{n!} \EE_{\mu \otimes \cdots \otimes \mu} \left[\det\left(f_i(X_j)\right)\det\left(g_i(X_j)\right)\right].
$$
\end{prop}

\begin{proof} An elementary formula for determinant asserts that: 
\begin{equation*} 
n! \det\left(\EE_\mu \left[f_i(X) g_j(X) \right]\right) = \sum_{\sigma,\sigma' \in S_n} \ve(\sigma) \ve(\sigma') \prod_{i=1}^n \EE_\mu \left[f_{\sigma(i)}(X) g_{\sigma'(i)}(X) \right].
\end{equation*}
Fubini's theorem allows us to write: 
\begin{equation*}
n! \det\left(\EE_\mu \left[f_i(X) g_j(X) \right]\right) = \sum_{\sigma,\sigma' \in S_n} \ve(\sigma) \ve(\sigma') \EE_{\mu \otimes \cdots \otimes \mu} \left[ \prod_{i=1}^n  f_{\sigma(i)}(X_i) g_{\sigma'(i)}(X_i) \right].
\end{equation*}
We thus have: 
\begin{eqnarray*}
n! \det\left(\EE_\mu \left[f_i(X) g_j(X) \right]\right) &=& \EE_{\mu \otimes \cdots \otimes \mu} \left[ \sum_{\sigma,\sigma' \in S_n} \ve(\sigma) \ve(\sigma') f_{\sigma(i)}(X_i) g_{\sigma'(i)}(X_i) \right] \\
&=& \EE_{\mu \otimes \cdots \otimes \mu} \left[ \left( \sum_{\sigma \in S_n} \ve(\sigma) f_{\sigma(i)}(X_i) \right) \left( \sum_{\sigma' \in S_n} \ve(\sigma') g_{\sigma'(i)}(X_i)\right) \right] \\
&=& \EE_{\mu \otimes \cdots \otimes \mu} \left[ \det\left(f_j(X_i)\right)\det\left(g_j(X_i)\right) \right] \\
&=& \EE_{\mu \otimes \cdots \otimes \mu} \left[ \det\left(f_i(X_j)\right)\det\left(g_i(X_j)\right) \right]. 
\end{eqnarray*}
\end{proof}
Note that with the particular choice $n=2$, $f_1=1$, $f_2=f$, $g_1=1$, $g_2=g$, Proposition~\ref{prop:Andreev}   gives the so-called "Chebyshev's other inequality":

\begin{prop}[Chebyshev] \label{prop:Chebyshev1d}
If $f,g \in L^2(\mu)$ are both non-increasing or both non-decreasing, then $\cov_\mu(f,g) \ge 0$. 
\end{prop}


\subsection{A first proof of Theorem~\ref{thm:main-1d}}

The first proof of Theorem~\ref{thm:main-1d} will be deduced from the following more general result.

\begin{theo} \label{th:CovaC}
Let $(u,U)$ and $(v,V)$ be two pairs of functions satisfying Assumption ($\mathcal{C}$). Then
\begin{equation*}
\cov_\mu(U,V) \cov_\mu(u,v) \ge \cov_\mu(u,V) \cov_\mu(U,v).
\end{equation*}
\end{theo}

\begin{proof} We want to show that $D \ge 0$, where 
\begin{equation*}
D = \left|\begin{array}{cc} \cov_\mu(u,v) & \cov_\mu(u,V) \\ \cov_\mu(U,v) & \cov_\mu(U,V) \end{array}\right|.
\end{equation*}
But we also have:
\begin{equation*}
D = \left|\begin{array}{ccc} 1 & \EE_\mu[v] & \EE_\mu[V] \\ \EE_\mu[u] & \EE_\mu[u v] & \EE_\mu[u V] \\ \EE_\mu[U] & \EE_\mu[U v] & \EE_\mu[U V]\end{array} \right|.
\end{equation*}
The latter equality can be proven by simply expanding the determinant. We now apply Proposition~\ref{prop:Andreev} with $f_1=1,f_2=u,f_3=U$ and $g_1=1,g_2=v,g_3=V$. This gives
\begin{equation*}
D = \int_{\RR^3} \left|\begin{array}{ccc} 1 & 1 & 1 \\ u(x_1) & u(x_2) & u(x_3) \\ U(x_1) & U(x_2) & U(x_3) \end{array} \right|\left|\begin{array}{ccc} 1 & 1 & 1 \\ v(x_1) & v(x_2) & v(x_3) \\ V(x_1) & V(x_2) & V(x_3) \end{array} \right| d\mu(x_1) d\mu(x_2) d\mu(x_3).
\end{equation*}
Let $(x_1,x_2,x_3) \in \RR^3$ and $\sigma \in S_3$ be such that $x_{\sigma(1)} < x_{\sigma(2)} < x_{\sigma(3)}$. As both pairs $(u,U)$ and $(v,V)$ satisfy ($\mathcal{C}$), we apply Proposition~\ref{prop:AssCsignD} as follows,
\begin{eqnarray*}
\left|\begin{array}{ccc} 1 & 1 & 1 \\ u(x_1) & u(x_2) & u(x_3) \\ U(x_1) & U(x_2) & U(x_3) \end{array} \right|\left|\begin{array}{ccc} 1 & 1 & 1 \\ v(x_1) & v(x_2) & v(x_3) \\ V(x_1) & V(x_2) & V(x_3) \end{array} \right| \\ = \left(\ve(\sigma) \left|\begin{array}{ccc} 1 & 1 & 1 \\ u(x_1) & u(x_2) & u(x_3) \\ U(x_1) & U(x_2) & U(x_3) \end{array} \right|\right) \left(\ve(\sigma)\left|\begin{array}{ccc} 1 & 1 & 1 \\ v(x_1) & v(x_2) & v(x_3) \\ V(x_1) & V(x_2) & V(x_3) \end{array} \right|\right) \ge 0,
\end{eqnarray*}
from which we deduce $D \ge 0$, as wanted. 
\end{proof}

\medskip

As we prove now, points (1) and (2) of Theorem~\ref{thm:main-1d} are particular cases of Theorem~\ref{th:CovaC}, for a suitable choice for the pairs $(u,U)$ and $(v,V)$.

\medskip

\begin{proof}[Proof of Theorem~\ref{thm:main-1d}(1) and (2)] 

The first item is a direct consequence of Theorem~\ref{th:CovaC} and Proposition \ref{prop:ConvC} for the particular choice $u(x)=x$, $U(x)=f(x)$, $v(x)=x$ and $V(x) = g(x)$.

For the second item, we set $u(x) = f(x)$, $U(x) = \int_0^x f(t) dt$, $v(x)=x$, $V(x)=g(x)$. Propositions~\ref{prop:logconcC} and~\ref{prop:ConvC} show that the pairs $(u,U)$ and $(v,V)$ both satisfy Assumption ($\mathcal{C}$). By Theorem~\ref{th:CovaC}, we thus have:
\begin{equation*}
\cov_\mu(f,x) \cov_\mu(U,g) \ge \cov_\mu(f,g) \cov_\mu(U,x).
\end{equation*}
The orthogonality assumption gives $\cov_\mu(f,x)=0$. Moreover, as $f$ is non-negative, the function $U$ is non-decreasing. Proposition~\ref{prop:Chebyshev1d} then gives $\cov_\mu(U,x) \ge 0$, so that the inequlity $\cov_\mu(f,g)\le 0$ holds, as desired. 

\end{proof}

\section{The  Hoeffding covariance identity approach in dimension one}
\label{sec:kernel-dim1}
 This section is devoted to a second proof of Theorem~\ref{thm:main-1d}(1) and (2). The proof will follow from the Hoeffding covariance identity and the use of the FKG inequality for a new probability measure on $\R^2$. A key point is that the kernel of  the Hoeffding representation is  \emph{totally positive}. 

\subsection{Hoeffding's covariance identity} 
We start by recalling the following representation formula for the covariance, which is a consequence of a slightly more general covariance identity due to Hoeffding. See \cite{saumwellner2017efron} for more details about Hoeffding's covariance identity.

\begin{theo}\label{thm:covA} Let $\mu$ be a probability measure
on $\R$ and denote by $F_{\mu}$ its cumulative distribution function,
then for all functions $f$ and $g$ in $L^2(\mu)$ and absolutely continuous, one has
\begin{equation}
\Cov_{\mu}(f,g)=\iint f'(x)k_\mu(x,y)g'(y)dxdy,\label{eq:cov-k}
\end{equation}
with 
\[
k_\mu(x,y)=F_{\mu}(x\wedge y)-F_{\mu}(x)F_{\mu}(y)
\]
and $x\wedge y=\min(x,y)$. \end{theo}
For simplicity, when it is clear from context, we shall write $k=k_\mu$ in the sequel.
We now recall some properties of the kernel $k:\R^{2}\to[0,+\infty)$.
Taking $f(\cdot)=\one_{[x,\infty[}(\cdot)$ and $g(\cdot)=\one_{[y,\infty[}(\cdot)$,
one sees that the kernel $k$ is necessarily unique and can also be
written 
\[
k(x,y)=\Cov_{\mu}\left(1_{\left\{ X\leq x\right\} },1_{\left\{ X\leq y\right\} }\right).
\]

This kernel is non-negative, bounded, continuous if $\mu$ is assumed
to be a continuous measure, but it is not differentiable on the line
$y=x$.
Let us emphasize the fact that the kernel $k$ is \emph{totally positive}
in the sense of Karlin \cite{karlin:book}. This result should be classical but we could not find a reference of it in
the literature.

\begin{theo}\label{thm:totpos} For all $n\geq2$, $s_{1}\leq\dots\leq s_{n}\in\R$
and $t_{1}\leq\dots\leq t_{n}\in\R$, 
\[
\det\left(k(s_{i},t_{j})\right)_{1\leq i,j\leq n}\geq0.
\]
\end{theo}

\begin{proof} The proof
follows from Theorem 3.1 in Karlin \cite{karlin:book}, or Theorem 4.2
in Pinkus \cite{pinkus:book}, by showing that the matrix $\left(k(s_{i},t_{j})\right)_{1\leq i,j\leq n}$
is a Green matrix. One can also directly use Corollary 3.1 in Karlin \cite{karlin:book}
by writing 
\[
k(x,y)=\left\{ \begin{array}{l}
\phi(x)\psi(y)\textrm{ if }x\geq y\\
\psi(x)\phi(y)\textrm{ if }x\leq y
\end{array}\right.
\]
with $\phi(x)=F(x)$ non-decreasing and $\psi(y)=1-F(y)$ non-increasing.
\end{proof}

\medskip
In the case $n=2$, Theorem \ref{thm:totpos} provides the following inequality.
\begin{coro} \label{cor:totpos2}
For all $s_{1}\leq s_{2}$ and $t_{1}\leq t_{2}$, 
\begin{equation}
k(s_{1},t_{1})k(s_{2},t_{2})\geq k(s_{1},t_{2})k(s_{2},t_{1}).\label{eq:totpos2}
\end{equation}
\end{coro}

The conclusion of Corollary \ref{cor:totpos2} is well known in the
literature under different names. Inequality \eqref{eq:totpos2},
here in the case of $\R^{2}$, is sometimes referred to as the \emph{Holley
condition} or the \emph{strong FKG condition}. The kernel $k$ is
also called \emph{log-supermodular} or \emph{multivariate totally
positive of order 2.}
We shall also the need the following extension:
\begin{coro}\label{cor:totpos}
Let $a$ and $b$ be two positive functions on $\R$ and define the kernel $k_{a,b}$ on $\R^2$ by
\[
k_{a,b}(x,y) = a(x) k(x,y) b(y).
\]
Then the kernel is totally positive and thus satisfies inequality \eqref{eq:totpos2}.
\end{coro}


We recall now the classical result, due to Fortuin, Kasteleyn  and Ginibre  \cite{FKG}, which asserts that the Holley condition 
implies the FKG inequality. We first state the definition of the FKG inequality in $\R^d$.

\begin{defi}\label{def:FKG} Let $d \geq 1$. A function $f:\R^{d}\to\R$
is said to be \emph{coordinate increasing} if it is non-decreasing
along each coordinate, that is if: 
\begin{equation}
\textrm{for all }x,y\in\R^{d},\textrm{ satisfying }x_{i}\leq y_{i},1\leq i\leq d \textrm{ one has }f(x)\leq f(y).\label{eq:croissante}
\end{equation}
A probability measure $\nu$ on $\R^{d}$ is said to satisfy the
\emph{FKG inequality} if for all functions $f$ and $g$ \emph{coordinate increasing}, one has: 
\begin{equation}
\Cov_{\nu}(f,g)\geq0.\label{e:def-FKG}
\end{equation}
\end{defi} 


\begin{theo}\label{thm:FKG} Let $\nu$ be a probability measure
$\R^{d}$ with density $k$ with respect to the Lebesgue measure.
Assume that for all $x,y\in\R^{d}$, 
\begin{equation}
k(x\wedge y)k(x\vee y)\geq k(x)k(y),\label{eq:k-FKG}
\end{equation}
where $x\wedge y=(\min(x_{1},y_{1}),\dots,\min(x_{d},y_{d}))$ and
$x\vee y=(\max(x_{1},y_{1}),\dots,\max(x_{d},y_{d}))$. Then $\nu$
satisfies the FKG inequality. \end{theo}


\begin{remark}\label{rmk:bakry-michel} Writing
$k=e^{H}$, the condition \eqref{eq:k-FKG} writes: 
\begin{equation}\label{eq:cond-Holley-H}
H(x\wedge y)+H(x\vee y)\geq H(x)+H(y).
\end{equation}
In the case where $k$ is smooth - more precisely when $H$ is of class
$\mathcal{C}^{2}$ here -, inequality \eqref{eq:cond-Holley-H} is equivalent to the following condition
on the second order cross-derivatives of $H$: 
\[
\frac{\partial^{2}}{\partial x_{i}\partial x_{j}}H(x)\geq0\textrm{ for }1\leq i\neq j\leq d.
\]
In this case, Bakry and Michel \cite{bakry-michel} proved the slightly
stronger result that the associated semi-group (see Section \ref{sec:comments})
preserves the class of coordinate increasing functions. Finally note that the kernel $k_{\mu}$ given in Theorem \ref{thm:covA} above, is not smooth on the diagonal of $\R^{2}$ and
that $\partial_{x,y}^{2}\ln k_\mu(x,y)=0$ for all $x\neq y\in\R^{2}$.
\end{remark}

\begin{remark}  Condition \eqref{eq:k-FKG} is not equivalent
to the FKG inequality. In the Gaussian setting, for a Gaussian vector
with non-singular matrix covariance $\Gamma$ on $\R^{d}$, the Holley
condition \eqref{eq:k-FKG} is equivalent to $(\Gamma^{-1})_{i,j}\leq0$
for $1\leq i\neq j\leq d$. But as proven by Pitt \cite{pitt:82}
in the Gaussian setting, the FKG inequality is equivalent to $\Gamma_{i,j}\geq0$
(see also Tong \cite{tong:book}). The condition on the coefficient
of $\Gamma^{-1}$ implies the one for $\Gamma$. But the converse
does not hold. The example 4.3.2 in Tong \cite{tong:book} provides
a covariance matrix for $d\geq3$ such that $\Gamma_{i,j}\geq0$ for
all $1\leq i,j\leq d$ but not $(\Gamma^{-1})_{i,j}\leq0$ for all
$1\leq i\neq j\leq d$.  
\end{remark}


\
\subsection{Hoeffding's formula as a relation between covariances}\label{subsec:nice-relation-dim1}
The main result of this section is Lemma \ref{thm:cov2} where  we express the  quantities appearing  in Theorem \ref{thm:main-1d} as a  covariance of the derivatives  of the functions with respect to a new probability measure on $\R^2$.
Let $\mu$ be a probability measure on $\R$ admitting a second moment.
We recall that $k$ is the non-negative kernel: 
\[
k(x,y)=F_{\mu}(x\wedge y)-F_{\mu}(x)F_{\mu}(y).
\]
and that from Theorem \ref{thm:covA}, it satisfies 
\begin{equation}
\iint k(x,y)dxdy=\Var_{\mu}(x)=\Var(\mu).\label{eq:varmu}
\end{equation}
By assumption, this last quantity is finite and we denote by $\mu^{(1)}$ the following probability measure on $\R^{2}$:

\[
d\mu^{(1)}(x,y)=\frac{k(x,y)}{\iint k(x',y')dx'dy'}dxdy
\]
In the case where $f$ and $g$ are some positive  functions, we also denote,
\[
d\mu_{f}^{(1)}(x,y)=\frac{f(x) k(x,y)}{\iint f(x') k(x',y')dx'dy'}dxdy
\]
and 
\[
d\mu_{f,g}^{(1)}(x,y)=\frac{f(x) k(x,y) g(y)}{\iint f(x') k(x',y')dx'dy'}dxdy
\]
The main result here  is the following relation between the  covariances of
$\mu$ and $\mu^{(1)}$. It consists essentially in a rewriting of Hoeffding's covariance identity \eqref{eq:cov-k} and to highlight the slight difference, we call it ``Hoeffding's covariance relation''.

 \begin{lem}[Hoeffding's covariance relation]\label{thm:cov2}
 Let $\mu$ be a probability measure on $\R$ admitting a second moment, with $\Var(\mu)>0$. Let $f,g:\R\to \R$ be some absolutely continuous functions  that belong to $L^2(\mu)$.
 
 \begin{enumerate} \item Then,

\begin{equation}
\frac{\Cov_{\mu}(f(x),g(x))}{\Var(\mu)}-\frac{\Cov_{\mu}(f(x),x)}{\Var(\mu)}\,\frac{\Cov_{\mu}(g(x),x)}{\Var(\mu)}=\Cov_{\mu^{(1)}}(f'(x),g'(y)).\label{eq:cov2}
\end{equation}
 \item If moreover, $f=e^{-\phi}$ is positive:
\begin{equation}
\frac{\Cov_{\mu}(f(x),g(x))}{Z_f}- \frac{1}{Z_f^2 }\Cov_{\mu}(f(x),x)\,\Cov_{\mu}(F(x),g(x))=\Cov_{\mu_f^{(1)}}(- \phi'(x),g'(y)).\label{eq:cov2b}
\end{equation}
 \item If moreover $g=e^{-\psi}$ is positive:
\begin{equation}
\frac{\Cov_{\mu}(f(x),g(x))}{Z_{f,g}}- \frac{1}{Z_{f,g}^2 }\Cov_{\mu}(f(x),G(x))\,\Cov_{\mu}(F(x),g(x))=\Cov_{\mu_{f,g}^{(1)}}( \phi'(x),\psi'(y)).\label{eq:cov2c}
\end{equation}
\end{enumerate} 
Here  $F$ and $G$ are primitives of $f$ and $g$ and 
\[
Z_f= \iint f(x)k(x,y) dxdy = \cov(F(x),x)>0,
\] 
and \[ Z_{f,g}= \iint f(x)k(x,y)g(y) dxdy = \cov(F(x),G(x))>0. 
\]
\end{lem}

Note that this approach is also linked to determinants in the sense that the left hand sides of the  equalities \eqref{eq:cov2}, \eqref{eq:cov2b} \eqref{eq:cov2c} can be written as determinants. For example, formula \eqref{eq:cov2} can be written as
\begin{equation}\label{formula-det}
\Var(\mu)^{2}\Cov_{\mu^{(1)}}(f'(x),g'(y))=\det\begin{pmatrix}\Var(\mu) & \cov_{\mu}(x,f(x))\\
\cov_{\mu}(x,g(x)) & \cov_{\mu}(f,g)
\end{pmatrix}.
\end{equation}

\
 \begin{proof} Using several times the covariance representation
of Theorem \ref{thm:covA}, one has: 
\begin{align*}
 & \frac{\Cov_{\mu}(f(x),g(x))}{\Var(\mu)}
=  \iint f'(x)\frac{k(x,y)}{\iint k(x',y')dx'dy'}g'(y)dxdy\\
= & \Cov_{\mu^{(1)}}(f'(x),g'(y))\\ 
 & +\left(\iint f'(x)\frac{k(x,y)}{\iint k(x',y')dx'dy'}dxdy\right)\left(\iint g'(y)\frac{k(x,y)}{\iint k(x',y')dx'dy'}dxdy\right)\\
= & \Cov_{\mu^{(1)}}(f'(x),g'(y))+\frac{\Cov_{\mu}(f(x),x)}{\Var(\mu)}\,\frac{\Cov_{\mu}(g(x),x)}{\Var(\mu)}.
\end{align*}
Similarly, if $f=e^{-\phi}$,
\begin{align*}
 & \frac{\Cov_{\mu}(f(x),g(x))}{Z_f}
= \iint( - \phi'(x) ) \frac{ f(x) k(x,y)}{Z_f}g'(y)dxdy\\
= & \Cov_{\mu_f^{(1)}}(- \phi'(x),g'(y))+ \frac{1}{Z_f^2 }\left(\iint f'(x)k(x,y)dxdy\right)\left(\iint f(x)k(x,y) g'(y)dxdy\right)\\
= & \Cov_{\mu_f^{(1)}}(- \phi'(x),g'(y))+ \frac{1}{Z_f^2 }\Cov_{\mu}(f(x),x)\,\Cov_{\mu}(F(x),g(x))
\end{align*}
and if moreover $g=e^{-\psi}$,
\begin{align*}
 & \frac{\Cov_{\mu}(f(x),g(x))}{Z_{f,g}}
= \iint( - \phi'(x) ) \frac{ f(x) k(x,y)g(y) }{Z_{f,g}} ( -\psi'(y))dxdy\\
= & \Cov_{\mu_{f,g}^{(1)}}( \phi'(x), \psi'(y))+ \frac{1}{Z_{f,g}^2 }\left(\iint f'(x)k(x,y) g(y) dxdy\right)\left(\iint f(x)k(x,y) g'(y)dxdy\right)\\
= & \Cov_{\mu_{f,g}^{(1)}}(\phi'(x),\psi'(y))+ \frac{1}{Z_{f,g}^2 }\Cov_{\mu}(f(x),G(x))\,\Cov_{\mu}(F(x),g(x)).
\end{align*}
\end{proof}

\subsection{A second proof of Theorem \ref{thm:main-1d}}

\label{sec:proof1d} We are now ready to turn to the second proof of Theorem \ref{thm:main-1d}(1) and (2)
pertaining to dimension one. The last ingredient will be the use of the FKG inequality.

\begin{proof}[Proofs of Theorem \ref{thm:main-1d}(1) and (2)] Let $\mu$
be any probability measure $\R$ admitting a second moment and let
$f$ and $g$ be two convex functions on $\R$. Using the  first covariance
relation of Lemma \ref{thm:cov2}, it is equivalent to prove that:
\[
\Cov_{\mu^{(1)}}(f'(x),g'(y))\geq0.
\]
By Corollary \ref{cor:totpos2} and Theorem \ref{thm:FKG}, the
probability measure $\mu^{(1)}$ on $\R^{2}$ satisfies the FKG inequality.
Since $f$ and $g$ are convex, the functions $(x,y) \to f'(x)$ and $(x,y)\to g'(y)$ on $\R^2$ are in particular increasing
along coordinates in $\R^{2}$, which implies the desired inequality. 

As for the second item, let $f=e^{-\phi}$ be a log-concave function and $g$ be a convex function on $\R$ and such  that $f$ is orthogonal to the linear function $x$.
By the second covariance formula  of  Lemma \ref{thm:cov2}, and since  $\cov_\mu(f(x),x)=0$, 
\[
\cov_\mu(f,g)= - \cov_{\mu^{(1)}_f} (\phi'(x), g'(y)).
\]
The second point follows similarly as above, since by Corollary \ref{cor:totpos} the kernel of the probability measure $\mu^{(1)}_f$ is also totally positive.
\end{proof}

\section{More covariance inequalities in dimension one for Chebyshev systems}
\label{sec:cheb} 

In this section, we consider some generalizations
in dimension one of Theorem \ref{thm:main-1d}, involving
formulations through determinants for Chebyshev systems. As for Theorem \ref{thm:main-1d}, we provide two proofs, one based only through determinantal identities and the other taking advantage of the strong fact that the
kernel $k_\mu$ is totally positive.

\subsection{Covariance inequalities for Chebyshev systems}
Let us first define Chebyshev systems.
\begin{defi} \label{def:Cheby} A $r$-uple of functions $(f_{1},\dots,f_{r})$ with $f_i : \R \to \R$ is said to
form a \emph{Chebyshev system} (of order r) if for all $t_{1}\leq\dots\leq t_{r}\in\R$,
\begin{equation}
\det(f_{i}(t_{j}))_{1\leq i,j\leq r}\geq0.\label{eq:cheb-def}
\end{equation}
\end{defi}


The main result of this section is the following:

\begin{theo}\label{thm:cheb} Let $n\geq1$ and $f_{1},\dots,f_{n}:\R\to\R$
and $g_{1},\dots,g_{n}:\R\to\R$ be some functions such that 
      both the $(n+1)$-uples $(1,f_1,\ldots,f_n)$ and $(1,g_1,\ldots,g_n)$ form two Chebyshev systems.
Denote $F(x)=\begin{pmatrix}f_{1}(x)\\
\dots\\
f_{n}(x)
\end{pmatrix}$ and $G(x)=\begin{pmatrix}g_{1}(x)\\
\dots\\
g_{n}(x)
\end{pmatrix}$, then: 
\begin{equation}
\det(\Cov(F,G))\geq0.\label{eq:cheb-concl}
\end{equation}
\end{theo}

\subsection{A first proof using the determinantal approach.}
We produce here a proof of Theorem \ref{thm:cheb} which is based on the methods introduced in Section~\ref{sec:Andreev-dim1}.

We first claim that
\begin{equation}\label{eq:egalite-det}
   \det\left(\cov_\mu(F ,G)\right)_{1 \le i,j \le n}= \det\left( \EE_\mu\left[f_i g_j  \right]_{0 \le i,j \le n}\right),
\end{equation}
where we set $f_0(x) = 1$ and $g_0(x)=1$. Indeed, let $A$ be the matrix $\left(\EE_\mu\left[f_i g_j \right]\right)_{0 \le i,j \le n}$. Let $C_0,\ldots,C_n$ denote the columns of the matrix $A$. Let us consider the matrix $B$ obtained by replacing, for every $1 \le j \le n$, the column $C_j$ by $C_j-\EE[g_j] C_0$. As $B$ is obtained from $A$ only by elementary operations, they have the same determinant. Moreover, $B$ can be written in block form:
\begin{equation*}
    B=\begin{pmatrix} 1 & 0 \\ E_\mu(F) & \Cov(F,G) \end{pmatrix}.
\end{equation*}
This proves that $\det(A) = \det\left(\Cov(F,G)\right)$ and \eqref{eq:egalite-det} follows. 

\begin{proof}[Proof of Theorem \ref{thm:cheb}]
By the equality \eqref{eq:egalite-det} and    the Andreev-type formula of Proposition~\ref{prop:Andreev}, one has 
$$
\det\left(\Cov(F,G)\right) = \frac{1}{n!} \int_{\R^{n+1}} \det\left(f_i(x_j)\right)_{0 \le i,j \le n} \det\left(g_i(x_j)\right)_{0 \le i,j \le n} d\mu(x_0)\cdots d\mu(x_n).
$$
Let us fix $(x_0,\ldots,x_n)\in \R^{n+1}$. There exists a permutation $\sigma \in S_n$ such that $x_{\sigma(0)} \le \cdots \le x_{\sigma(n)}$. As $(f_0,\ldots,f_n)$ and $(g_0,\ldots,g_n)$ are Chebyshev systems, one has
$$
\varepsilon(\sigma) \det\left(f_i(x_j)\right)_{0 \le i,j \le n} \ge 0 \ , \ \varepsilon(\sigma) \det\left(g_i(x_j)\right)_{0 \le i,j \le n} \ge 0.
$$
One then has $\det\left(f_i(x_j)\right)_{0 \le i,j \le n}\det\left(g_i(x_j)\right)_{0 \le i,j \le n} \ge 0$, so $$\det\left(\Cov(F,G)\right) \ge 0$$
and the result follows.
\end{proof}

\subsection{A second proof with the Hoeffding covariance identity.}
We turn now to the proof of the covariance inequality of Theorem \ref{thm:cheb} using Hoeffding's covariance identity \eqref{eq:cov-k}. The main point will be the use of a  bivariate Andreev-type formula for bilinear integral operators.
Another key point is to transfer the assumption on the functions to an assumption on their derivatives. 
\begin{prop} \label{prop:equiv-cheb}
 Let $n\geq1$ and $f_{1},\dots,f_{n}:\R\to\R$
 be some $\mathcal C^1$ functions.  The following assertions are equivalent:
\begin{enumerate}
    \item The $(n+1)$-uple $(1,f_1,\ldots,f_n)$ forms a  Chebyshev system.
    \item The  $n$-uple $(f_{1}',\dots,f_{n}')$
forms a  Chebyshev system.
\end{enumerate} 
\end{prop}

\begin{proof}
Assume (1), let $x_0 <x_1 <\cdots <x_n$  and set $f_0 =1$. Then, replacing the  $i$-th column $C_i$ by $C_i- C_{i-1}$ for $1\leq i \leq n$,  yields 
\begin{align*}
   0\leq  \det(f_i (x_j) )_{0\leq i,j\leq n} &= \det(f_i (x_{j}) -f_i(x_{j-1})  )_{1\leq i,j\leq {n}}\\
    &=  \prod_{j=1}^n (x_{j}-x_{j-1}) \det \left( \frac{f_i (x_{j}) -f_i(x_{j-1})} {x_{j}-x_{j-1}} \right)_{1\leq i,j\leq {n}}.
\end{align*}
Letting successively $x_{j}$ tend to $x_{j-1}$ for $j=1,...,n$ gives 
\[
\det(f_i' (x_{j-1}))_{1\leq i,j\leq {n}} \geq 0
\]
and (2) follows.

Now assume (2), for $x_0<x_1 <\cdots <x_n$. Since one has $f_0=1$, replacing the  $i$-th column $C_i$ by $C_i- C_{i-1}$ for $1\leq i \leq n$, one gets
\begin{align*}
\det(f_i (x_j) )_{0\leq i,j\leq n} &=  \det(f_i (x_j) -f_i(x_{j-1})  )_{1\leq i,j\leq n}\\
&=   \det\left( \int_{x_{j-1}}^{x_j}  f_i' (u_j) d u_j  \right)_{1\leq i,j\leq n}\\
&=  \int_{x_{n-1 }}^{x_{n}}  \dots  \int_{x_{1}}^{x_2}  \int_{x_{0}}^{x_1}  \det(f_i' (u_j))_{1\leq i,j\leq {n}} \;   du_1 du_2 \dots du_n 
\end{align*}
where the last line follows from an Andreev-type formula. Since $u_1 \leq u_2 \leq  \cdots \leq u_n$, we have  
$\det(f_i' (u_j)) \geq 0$ by (2), and (1) follows.
\end{proof}

\

The second key point is the following bivariate
Andreev-type  formula for bilinear kernel integral operators on $\R^{2n}$. It is  stated here without a proof.

\begin{prop}\label{prop:bivar-andreev}
With the same notation as in Theorem \ref{thm:cheb}, 
\begin{align*}
\det(\Cov(F,G)) & =\det\left(\iint_{x,y\in\R}f'_{i}(x)k(x,y)g'_{j}(y)d\mu(x)d\mu(y)\right)_{1\leq i,j\leq n}\\
 & =\iint_{\mathcal{D}}\det\left(f_{i}'(x_{j})\right)\det\left(k(x_{i},y_{j})\right)\det\left(g_{i}'(y_{j})\right)
 dx_1 \dots dx_d  dy_1 \dots dy_d
\end{align*}
where $\mathcal{D}$ is the
domain of $\RR^{2n}$ defined by 
\[
\mathcal{D}=\left\{ (x_{1},\ldots,x_{n},y_{1},\ldots,y_{n})\in\RR^{2n}\ :\ x_{1}<\cdots<x_{n}\ ,\ y_{1}<\cdots<y_{n}\right\} .
\]
\end{prop} 
A second  proof of Theorem \ref{thm:cheb} is then immediate. 
\begin{proof}[Second proof of Theorem \ref{thm:cheb}] By hypothesis and since the kernel $k$ is totally positive,
the three determinants in the integral on the second line in the equality of  Proposition \ref{prop:bivar-andreev}  are non-negative and the result follows.
\end{proof}

\subsection{Some applications.}
We shall use Theorem \ref{thm:cheb} under the following particular
form.
\begin{coro}\label{cor:cheb} Let $n\geq1$ and $\phi_1,\dots,\phi_{n}:\R\to\R$
and $f,g:\R\to\R$ be some functions and denote $F(x)=\begin{pmatrix}\phi_{1}(x)\\
\dots\\
\phi_{n}(x)\\
f(x)
\end{pmatrix}$ and $G(x)=\begin{pmatrix}\phi_{1}(x)\\
\dots\\
\phi_{n}(x)\\
g(x)
\end{pmatrix}$. Assume that $(1,\phi_{1},\dots,\phi_{n},f)$ and $(1,\phi_{1},\dots,\phi_{n},g)$
form two Chebyshev systems, then: 
\begin{equation}
\det(\Cov(F,G))\geq0.\label{eq:cheb-concl2}
\end{equation}
\end{coro}

It is well known that, if $f$ is smooth and if  we choose more precisely $\phi_{1}(x)=x,\dots,\phi_{n}(x)=x^{n}$,
the condition that $(1,\phi_{1},\dots,\phi_{n},f)$ forms a Chebyshev
system  is 
equivalent to $f^{(n+1)}(x) \geq 0,\textrm{ for all }x\in\R$.
This is a well known generalization of Proposition \ref{prop:ConvC}, see \cite[Chapter 6 Example 4]{karlin:book}.

In the case where we consider $n=1$, we recover Theorem \ref{thm:main-1d}(1).
Indeed, in dimension one,  by Proposition \ref{prop:ConvC}, the convexity assumptions on  $f$ and $g$  are  equivalent to  the fact that $(1,x,f)$ and $(1,x,g)$  both form a Chebyshev system.



\

We now describe the result for $n=2$ and $\phi_{1}(x)=x,\phi_{2}(x)=x^{2}$.


\begin{coro}\label{cor:n=00003D3} Let $\mu$ be probability measure
on $\R$ admitting a fourth moment. Assume that $f^{(3)}(x)\geq0$
and $g^{(3)}(x)\geq0$ for all $x\in\R$, then 
\begin{align*}
\cov_{\mu}(f,g)\left(\Var_{\mu}(x^{2})\Var_{\mu}(x)-\cov_{\mu}(x,x^{2})^{2}\right)\\
\geq\begin{pmatrix}\cov_{\mu}(x,f) & \cov_{\mu}(x^{2},f)\end{pmatrix}\begin{pmatrix}\Var_{\mu}(x^{2}) & -\cov_{\mu}(x,x^{2})\\
-\cov_{\mu}(x,x^{2}) & \Var_{\mu}(x)
\end{pmatrix}\begin{pmatrix}\cov_{\mu}(x,g)\\
\cov_{\mu}(x^{2},g)
\end{pmatrix}.
\end{align*}

If moreover $\int_{\R}xd\mu=\int_{\R}x^{3}d\mu=0$, the latter inequality writes
\begin{align*}
\cov(f,g)\geq\frac{1}{\var_{\mu}(x)}\cov(x,f)\cdot\cov(x,g)+\frac{1}{\var_{\mu}(x^{2})}\cov(x^{2},f)\cdot\cov(x^{2},g).
\end{align*}
\end{coro} 

\

\section{The tensorization method for product measures}\label{sec:tens} 
We investigate here product measures on $\R^d$ with $d\geq 2$, through the use of 
the tensorization argument. This method consists in decomposing the covariance for the product measure by the one-dimensional   covariances of the marginals and then applying  the covariance inequalities previously obtained in dimension one.

\subsection{The tensorization decomposition of the covariance}
\begin{lem}\label{lem:usual-tens}
Let $\mu=\mu_{1}\otimes\dots\otimes\mu_{d}$ be a product measure.
For a function $f:\R^{d}\to\R$, set  
\[
f_{k}(x_{1},\dots x_{k})=\iint f(x_{1},\dots x_{d})d\mu_{k+1}(x_{k+1})\dots d\mu_{d}(x_{d}),
\]
for $1\leq k\leq d$, and set $f_{0}=\int fd\mu$. 
Then it holds,

\begin{equation}\label{eq:tensor}
 \cov_{\mu}(f,g)=\sum_{k=1}^{d}\iint\cov_{\mu_{k}}(f_{k},g_{k})d\mu_{1}(x_{1})\dots d\mu_{k-1}(x_{k-1}).
\end{equation}
\end{lem}

In the above lemma, the function $f_k$ is the conditional expectation of $f$ knowing $(x_1,\dots,x_k)$ and $\cov_{\mu_{k}}(f_{k},g_{k})$ is the covariance with respect to the one-dimensional marginal $\mu_k$ of $f_k$ and $g_k$ ; that  is the function depending on $(x_1,\dots, x_{k-1})$ given by:
\[
\cov_{\mu_{k}}\big(x\mapsto f_{k}(x_1,\dots, x_{k-1}, x) , x\mapsto g_{k}(x_1,\dots, x_{k-1}, x)\big).
\]

This decomposition of the covariance  is well known and implies  the famous tensorization  property of the Poincar\'e inequality; see e.g. \cite{bakry-gentil-ledoux}. To be complete,  \eqref{eq:tensor} is stated for the variance in \cite{ledoux-spin}, but it also applies to the covariance due to the following polarization identity:
\[
4 \, \cov(f,g) = \Var(f+g) -\Var (f-g).
\]

\subsection{A weighted Hoeffding's covariance relation in dimension one.}
For the sequel, we shall need a slight generalization of the covariance relation of Lemma~\ref{thm:cov2} in dimension one, that we describe now.

Let $\mu$ be a probability measure on $\R$. Let $a$ be a positive function on $\R$ and let $A$ be the (centered) primitive of $a$, $A=\int a dx+c$. We assume that $a$ is such that $\Var_{\mu}(A)<+\infty$.
With these notations and  with also the same notations as in Section \ref{subsec:nice-relation-dim1}, we define 
\[
k_{a,a}(x,y):={a(x)}k(x,y){a(y)}
\]
and we set $Z_{a,a}=\iint k_{a,a}(x,y)dxdy.$ By Theorem \ref{thm:covA}, one gets
$
Z_{a,a}=\Var_{\mu}(A).
$ We will also consider the measure $\mu_{a,a}^{(1)}$, defined by
\[
d\mu_{a,a}^{(1)}(x,y)=\frac{k_{a,a}(x,y)}{\iint k_{a,a}(x',y')dx'dy'}dxdy.
\]
Since the kernel $k$ is totally positive,  by Corollary \ref{cor:totpos2},  one gets that the kernels $k_{a,a}$ are also totally positive.

The following lemma provides a generalization of Lemma \ref{thm:cov2}, which corresponds to the case $a\equiv1$.

\begin{lem}\label{thm:cov-mod-a} 
Let $f,g:\R\to \R$ be some absolutely continuous functions  that belong to $L^2(\mu)$.
 
 \begin{enumerate} 
 \item Then,
\begin{align}\label{weight_cov_dim1_conv_conv}
 \nonumber & \cov_{\mu}(f(x),g(x))\\
 =&Z_{a,a}\left[\cov_{\mu_{a,a}^{(1)}}\left(\frac{f'(x)}{a(x)},\frac{g'(y)}{a(y)}\right)+\frac{\Cov_{\mu}(f(x),A(x))}{Z_{a,a}}\,\frac{\Cov_{\mu}(g(x),A(x))}{Z_{a,a}}\right]. 
\end{align}

\item If moreover $ f=e^{-\phi}$, then
\begin{align}\label{weight_cov_dim1_logco_conv}
 \nonumber &   \cov_{\mu}(f(x),g(x))\\
 =&
 Z_{af,a}\left[\cov_{\mu_{af,a}^{(1)}}\left(\frac{-\phi'(x)}{a(x)},\frac{g'(y)}{a(y)}\right)+\frac{\Cov_{\mu}(f(x),A(x))}{Z_{af,a}}\,\frac{\Cov_{\mu}(F_a(x),g(x))}{Z_{af,a}}\right].
\end{align}
\item If moreover $g=e^{-\psi}$, then
\begin{align} \label{weight_cov_dim1_logco_conv2}
\nonumber & \cov_{\mu}(f(x),g(x))\\
=&
 Z_{af,ag}\left[\cov_{\mu_{af,ag}^{(1)}}\left(\frac{\phi'(x)}{a(x)},\frac{\psi'(y)}{a(y)}\right)+\frac{\Cov_{\mu}(f(x),G_a(x))}{Z_{af,ag}}\,\frac{\Cov_{\mu}(F_a(x),g(x))}{Z_{af,ag}}\right].  
\end{align}
\end{enumerate} 
Here $F_a$ and $G_a$ denotes respectively the primitives of $ af $ and $ag$ and the constants $Z_{af,a}$ and $Z_{af,ag}$ are defined as in Lemma \ref{thm:cov2}. 
\end{lem}

\subsection{Proofs of Theorems \ref{thm:Hu-produit-A}, \ref{thm:Hu-tens-incond-A} and \ref{thm:Harge-Royen-tens} }

We first consider the proof of Theorem \ref{thm:Hu-produit-A}. We then state a similar result in Theorem \ref{thm:tens-V}, but with slightly different assumptions.

\begin{proof}[Proof of Theorem \ref{thm:Hu-produit-A}] First, we assume that $f$ and $g$ are such that all
the quantities in \eqref{eq:cond-a-i} and \eqref{eq:cond-a-ij} are non-negative. 
By the tensorization of the covariance  \eqref{eq:tensor} and the first covariance relation of Lemma \ref{thm:cov-mod-a}, one has
\begin{align}
\nonumber & \Cov_{\mu}(f,g)\\
= & \sum_{k=1}^{d}\frac{1}{Z_{k,a_k,a_{k}}}\iint\cov_{\mu_{k}}(f_{k},A_{k}(x_{k}))\cov_{\mu_{k}}(g_{k},A_{k}(x_{k}))d\mu_{1}(x_{1})\dots d\mu_{k-1}(x_{k-1})\label{l1}\\
 +& \sum_{k=1}^{d}Z_{k,a_k,a_{k}}\iint 
\cov_{\mu_{k,(a_k,a_{k})}^{(1)}}\left(  \frac{\partial_{k}f_{k}(x_{k})}{a_{k}(x_{k})} ,\frac{\partial_{k}g_{k}(y_{k})}{a_{k}(y_{k})}\right)
 d\mu_{1}(x_{1})\dots d\mu_{k-1}(x_{k-1}),\label{l2}
\end{align} 
where more precisely
\[
\partial_{k}f_{k}(x_{k}) =\partial_{k}f_{k}(x_{1},\dots,x_{k-1},x_{k}) \textrm{ and } \partial_{k}g_{k}(y_{k}) = \partial_{k}g_{k}(x_{1},\dots,x_{k-1},y_{k}).\]

We first prove that the sum in display \eqref{l2} is non-negative.
The assumption \eqref{eq:cond-a-i} implies that 
 both $(x_k,y_k)\mapsto \frac{\partial_{k}f(x_{k})}{a_{k}(x_{k})}$ and $(x_k,y_k)\mapsto \frac{\partial_{k}g(y_{k})}{a_{k}(y_{k})}$ are coordinatewise increasing on $\R^2$
for every fixed $x_{1},\dots,x_{k-1}$. Since by Corollary \ref{cor:totpos}, the measure $\mu_{k,(a_k,a_{k})}^{(1)}$
satisfies the FKG criterion on $\R^{2}$, the term \eqref{l2} is
non-negative.
We now turn to the sum in display \eqref{l1}. With a similar notation
for $g$, we set 
\begin{align*}
F_{k,a_{k}}(x_{1},\dots x_{k-1})&=\cov_{\mu_{k}}(f_{k},A_{k}(x_{k})).\\
\end{align*}
Since the hypotheses allow to exchange  derivation and integrals, one has 
for $1\leq i \leq (k-1)$, 
\begin{align*}
& \partial_i \, \cov_{\mu_{k}}(f_{k},A_{k}(x_{k}))\\
=& \iint_{x_k,y_k}   \left( \iint_{x_{k+1}, \dots,x_d} \partial_{i,k} f(x) d\mu_{k+1} (x_{k+1})\dots d\mu_d(x_d) \right) k_{\mu_k}(x_k,y_k)  a_k(y_k) d x_k dy_k.
\end{align*}
In particular,
the assumptions \eqref{eq:cond-a-ij} give that  the above integrands are non-negative. This implies that the functions $F_{k,a_{k}}$
and $G_{k,a_{k}}$ are both coordinate increasing on $\R^{k-1}$.
By the standard FKG inequality for product measures, one gets that 
\begin{align*}
 & \iint F_{k,a_{k}}\;G_{k,a_{k}}d\mu_{1}(x_{1})\dots d\mu_{k-1}(x_{k-1})\\
\geq & \iint F_{k,a_{k}}d\mu_{1}(x_{1})\dots d\mu_{k-1}(x_{k-1})\cdot\iint G_{k,a_{k}}d\mu_{1}(x_{1})\dots d\mu_{k-1}(x_{k-1})\\
= & \;\cov_{\mu}(f(x),A_{k}(x_{k}))\cdot\cov_{\mu}(g(x),A_{k}(x_{k})).
\end{align*}
Summing over the index $k$ ends the proof in the case of non-negative signs in assumptions \eqref{eq:cond-a-i} and \eqref{eq:cond-a-ij}.
Finally, analyzing the above proof, one sees that it is still valid for  general signs.
Indeed, under the general case of assumption \eqref{eq:cond-a-i}, the functions $(x_k,y_k)\mapsto \frac{\partial_{k}f(x_{k})}{a_{k}(x_{k})}$ and $(x_k,y_k)\mapsto \frac{\partial_{k}g(y_{k})}{a_{k}(y_{k})}$ are either both coordinatewise increasing 
or both  coordinatewise decreasing  on $\R^2$ and thus have a non-negative covariance with respect to the measure  $\mu_{k,(a_k,a_{k})}^{(1)}$.  Secondly, the last argument relies  on  the FKG inequality for product measures. In this case, the FKG inequality is in fact also valid   if the  functions $F_{k,a_k}$ and $G_{k,a_k}$ are monotone along coordinates, with the same monotonicity along each coordinate. This is the case under the general assumption \eqref{eq:cond-a-ij} and the result follows.

\end{proof}

As announced in the beginning of this section, we also obtain with this tensorization approach, a similar result under slightly different conditions.

\begin{theo}\label{thm:tens-V} 
Let $\mu$ be a product measure on $\R^{d}$. Assume that the marginals $\mu_k$ are absolutely continuous with respect to the Lebesgue measure, with positive densities $e^{-V_k}$, for smooth potentials $V_k$.
Let $f,g: \R^d \to \R$ and assume
 that for each $1\leq k\leq d$, the signs of 
\begin{equation}\label{cond-V-i}
\partial_{k}\left(\frac{\partial_{k}f(x)}{a_{k}(x_{k})}\right)\textrm{ and }\partial_{k}\left(\frac{\partial_{k}g(x)}{a_{k}(x_{k})}\right)
\end{equation}
are constant and equal and that for each   $(j,k)$ with $1\leq j<k\leq d$, the
signs of 
\begin{equation}\label{cond-V-ij}
\partial_{j,k}\left(f(x)\frac{A_{k}(x_{k})}{V_{k}'(x_{k})}\right)\textrm{ and }\partial_{j,k}\left(g(x)\frac{A_{k}(x_{k})}{V_{k}'(x_{k})}\right)
\end{equation}
are also constant and equal. 
Assume furthermore the following technical assumption:
\begin{equation} \label{eq:HypoTech}
\lim_{x_k \to \pm \infty} f_k(x_1,\ldots,x_k) \frac{A_k(x_k)}{V'_k(x_k)} e^{-V_k(x_k)} = 0,
\end{equation}
where $f_k$ is defined in Lemma~\ref{lem:usual-tens}. 
Then, it holds
\[
\cov_{\mu}(f,g)\geq\sum_{k=1}^{d}\frac{1}{\Var_{\mu_{k}}(A_{k})}\cov_{\mu}(f(x),A_{k}(x_{k}))\cdot\cov_{\mu}(g(x),A_{k}(x_{k})).
\]
\end{theo}

Note that  since each $A_k$ vanishes exactly  in one point, the hypothesis in \eqref{cond-V-ij} forces each $V_k$ to be  unimodal (in the sense that  $V_k$ has only one zero).
In particular, if each potential $V_k$ is strictly convex, one can specialize the result to the case $A_k=V_k'$ or equivalently $a_k=V_k''$. 
With this specific choice, Theorems  \ref{thm:tens-V} and  \ref{thm:Hu-produit-A} coincide.


\begin{proof}[Proof of Theorem \ref{thm:tens-V}]
The proof is similar to the one of Theorem \ref{thm:Hu-produit-A}. The only difference is that we use integration by parts to get a different representation of  $F_{k,a_{k}}$. We obtain
\[
F_{k,a_{k}}(x_{1},\dots x_{k-1})=\cov_{\mu_{k}}\left(f_{k},V_{k}'(x_{k})\frac{A_{k}(x_{k})}{V_{k}'(x_{k})}\right)=\int\partial_{k}\left(f_{k}\frac{A_{k}(x_{k})}{V_{k}'(x_{k})}\right)d\mu_{k}(x_{k}).
\]
Notice that the bracket terms in the integration by parts are zero due to Assumption~\eqref{eq:HypoTech}.
Furthermore, by Assumption \eqref{cond-V-ij}, the two functions $F_{k,a_k}$ and $G_{k,a_k}$ are coordinatewise monotone with the same kind of monotony along each coordinate. The result follows.
\end{proof}

\

We now add the symmetry assumptions and turn to the proofs of Theorem \ref{thm:Hu-tens-incond-A}.
and \ref{thm:Harge-Royen-tens}.

\begin{proof}[Proof of Theorem \ref{thm:Hu-tens-incond-A}]  We start
by the same covariance formulae, given in \eqref{l1}, \eqref{l2}, as in the proof of Theorem  \ref{thm:Hu-produit-A}. As previously, in view of assumption \eqref{eq:cond-l2-idem}, all terms in  \eqref{l2} are non-negative. Now, we show that under the symmetry assumptions, all the terms in \eqref{l1} vanish.
Indeed, since the product measure $\mu$ is symmetric, the function  $f$ is unconditional and the function $a_k$ are even, we have that, for any $1\leq k \leq d$, and any $x_1,\dots x_{k-1}$, the functions
\[
x_k \mapsto f_k(x_1,\dots x_{k-1},x_k)
\]
are even and that the primitive functions $A_k$ are odd.
Hence, it holds that
\[
\cov_{\mu_k}(f_k ,A_k(x_k)) =0  \textrm{ for each } 1\leq k \leq d \textrm{ and all } x \in \R^d
\]
and all the terms in \eqref{l1} vanish. The result follows.
\end{proof}

\begin{proof}[Proof of Theorem \ref{thm:Harge-Royen-tens}(1)]
With the same notations as above, writing $\phi_k=-\ln f_k$ and using the second point of Lemma \ref{thm:cov-mod-a}, one has
\begin{align*}
  & \Cov_{\mu}(f,g)\\%
  =&\sum_{k=1}^{d} \iint  \frac{1}{Z_{k,f_k,1}} \cov_{\mu_{k}}(f_{k},x_{k})
\left(\iint f_k(x_k)k_{\mu_k}(x_k,y_k) \partial_k g_k(y_k)dx_kdy_k\right)
d\mu_{1}\dots d\mu_{k-1}   \\ 
  +&\sum_{k=1}^{d}\iint Z_{k,f_k,1}\cov_{\mu_{k,(f_k, 1)}^{(1)}}\left(  -\partial_{k}\phi_{k}(x_{k}) ,\partial_{k}g_{k}(y_{k})\right)d\mu_{1}\dots d\mu_{k-1}.  
\end{align*}
Since $f$ is unconditional, for each fixed $(x_1,\dots,x_{k-1})$, the function $f_k$ is even, and thus
\[
\cov_{\mu_{k}}(f_{k},x_{k})=0
\] 
and the terms of the sum in the right-hand side of the above equality vanish.
Furthermore, as $f$ and the $\mu_i$ are log-concave, by stability through marginalization  of log-concavity (Prékopa's theorem), the functions $f_k$ are log-concave, meaning that the functions $\phi_k$ are convex. 
Since  by Corollary \ref{cor:totpos}, the measures   $\mu_{k,(f_k, 1)}^{(1)}$ satisfy the FKG inequality on $\R^2$, one has
for each fixed $(x_{1},\dots,x_{k-1})\in \R^{k-1}$,
\[
\cov_{\mu_{k,(f_k, 1)}^{(1)}}\left(  -\partial_{k}\phi_{k}(x_{1},\dots,x_{k-1},x_{k}) ,\partial_{k}g_{k}(x_{1},\dots,x_{k-1},y_{k})\right)\leq 0
\]
and the proof is complete.
\end{proof}
The proof of Theorem \ref{thm:Harge-Royen-tens}(2) is given in the next section. 
One can note that we do not state a version of Theorem \ref{thm:Harge-Royen-tens} with the function $a_k$ or $A_k$. The reason is that we do not know a natural hypothesis on $f$ that would induce a sign for the quantities $\partial_k \left(\frac{\partial_k \phi_k}{a_k} \right)$.

\section{The quasi-concave case}\label{sec:quasi-concave}

This section is devoted to the proof of Theorem \ref{thm:main-1d}(3) and  Theorem \ref{thm:Harge-Royen-tens}(2), related to the quasi-concave case. This assumption indeed requires different techniques than in the rest of the paper.  The result in dimension one is obtained through the so-called layer-cake representation of the functions (see \eqref{eq:quasi-concave}) and the result in dimension $d\geq 2$  is then obtained by tensorisation. A similar result already appears in \cite{Schechtman}, but  as far as we know,  the statement of  Theorem \ref{thm:Harge-Royen-tens}(2) is new.

\

Recall that a quasi-concave function $f$ on $\RR^d$ is a real-valued function that satisfies, for any $x,y\in \RR^d$ and any $\lambda \in [0,1]$,
\[
f(\lambda x + (1-\lambda)y)\geq \min \left\{ f(x), f(y)\right\}.
\]

An equivalent formulation of quasi-concavity consists in requiring that the upper level sets of the function are convex. 
In the following, we make use of a weaker notion than quasi-concavity, that we term ``coordinatewise quasi-concavity'': 
\begin{defi}
A function $f:\R^d \to \R$ is said to be \emph{coordinatewise quasi-concave} if 
 for all $(x_1,...,x_{i-1},x_{i+1},...,x_d)\in \RR^{d-1}$, the functions
 \[
 x_i \in \R \mapsto f(x_1,...,x_{i-1},x_i,x_{i+1},...,x_d) \in \R
 \]
 are quasi-concave.
\end{defi}

 Another characterization is thus that  for any $\lambda \in [0,1]$, any $(x_1,...,x_{i-1},x_{i+1},...,x_d)\in \RR^{d-1}$ and any $x,y\in \R$,
\begin{align*}
    & f(x_1,...,x_{i-1},\lambda x_i + (1-\lambda)y_i,x_{i+1},...,x_d)\\
& \geq \min  \left\{ f(x_1,...,x_{i-1},x_i,x_{i+1},...,x_d), f(x_1,...,x_{i-1}, y_i,x_{i+1},...,x_d)\right\}.
\end{align*}

The above definition and its characterization directly imply that  quasi-concave functions are coordinatewise quasi-concave, but the converse is not true.

Note that the interpretation in terms of convex upper level sets does not hold anymore for the notion of coordinatewise quasi-concavity. But still, the upper level sets of a coordinatewise quasi-concave function are connected sets.

We now turn to the proof of Theorem \ref{thm:main-1d} (3) in dimension one.
\begin{proof}[Theorem \ref{thm:main-1d} (3)] 
Let $f$ and $g$ be two non-negative quasi-concave even function on $\R$.
We write, for $x\in \R$,
\begin{equation}\label{eq:quasi-concave}
    f(x)= \int_0^\infty \bone_{A_t}(x) dt \textrm{ and }  g(x)= \int_0^\infty \bone_{B_t}(x) dt 
\end{equation}
where $A_t$ and $B_t$ for $t\geq 0$ are the the level sets of $f$ and $g$ defined by
\[
A_t:= \{x\in \R, f(x)\geq t \} \textrm{ and } B_t:= \{x\in \R, f(x)\geq t\} .
\]
The key point is here that since $f$ and $g$ are quasi-concave and even, the sets $A_s$ and $B_t$ are symmetric intervals on $\R$ and therefore, for each $s,t \geq 0$, $A_s\subset B_t$ or $B_t \subset A_s$.
Therefore by Fubini-Tonelli, one has 
\begin{align*}
    \int f(x) g(x) d\mu(x) &= \int_x \int_{s=0}^\infty \int_{t=0}^\infty  \bone_{A_s}(x)  \bone_{B_t}(x) ds dt d\mu(x)\\
    &= \int_{s=0}^\infty \int_{t=0}^\infty  \mu (A_s \cap B_t) ds dt\\
    &= \int_{s=0}^\infty \int_{t=0}^\infty \min( \mu (A_s),  \mu(B_t)) ds dt\\
    & \geq   \int_{s=0}^\infty \int_{t=0}^\infty  \mu (A_s) \mu  (B_t) ds dt\\
    &= \int  f(x) d\mu(x) \; \int g(x) d\mu(x); 
\end{align*}
which is precisely the desired inequality.
\end{proof}

\

We  now prove  Theorem \ref{thm:Harge-Royen-tens}(2)  by the tenzorisation method. The main argument is ensured by the following lemma, which states the stability of unconditional coordinatewise quasi-concavity by marginalization. Its proof can be found below.
\begin{lem}\label{lem:marg_quasi_con}
Consider an integer $d\geq 2$ and take $k\in\left\{1,...,d-1\right\}$. Assume that a function $f$ on $\RR^d$ is unconditional and coordinatewise quasi-concave. Then the function
\[
f_{k}(x_{1},\dots x_{k})=\int f(x_{1},\dots x_{d})d\mu_{k+1}(x_{k+1}) \dots d\mu_{d}(x_d)
\]
is coordinatewise quasi-concave and unconditional.
\end{lem}

\begin{proof}[Proof of  Theorem \ref{thm:Harge-Royen-tens}(2)]
Let $f$ and $g$ be unconditional and coordinatewise quasi-concave functions.
Let us first recall the standard  tensorization formula:
\[
\cov_{\mu}(f,g)=\sum_{k=1}^{d}\iint\cov_{\mu_{k}}(f_{k},g_{k})d\mu_{1}\dots d\mu_{k-1}.
\]
By Lemma \ref{lem:marg_quasi_con} above, for any $(x_1,...,x_{k-1})\in \R^{k-1}$ the functions $f_k(x_1,...,x_{k-1}, \cdot)$ and $g_k(x_1,...,x_{k-1}, \cdot)$ are even and quasi-concave on $\RR$.
By Theorem \ref{thm:main-1d} (3), one has $\cov_{\mu_{k}}(f_k,g_k) \geq 0$ and the result follows.
\end{proof}

\begin{proof}[Proof of Lemma \ref{lem:marg_quasi_con}]
Unconditionality of $f_k$ directly follows from unconditionality of $f$. As for the coordinatewise quasi-concavity, we will make the reasoning for the first coordinate $x_1$ and the arguments readily extend to the other coordinates. Take a pair $(x_1,y_1)$ such that $|x_1|\leq|y_1|$. Assume without loss of generality that $y_1\geq0$ (otherwise replace it by $-y_1$). As $f$ is unconditional and coordinatewise quasi-concave, for any $y\in[-y_1,y_1]$ and for any $(x_2,...,x_d)\in \RR^{d-1}$, we have
\begin{align*}
   f(y,x_2,\dots,x_d) & \geq  \min \left\{ f(-y_1,x_2,...,x_d), f(y_1,x_2,...,x_d)\right\} \\
                    & =  f(y_1,x_2,...,x_d) ,
\end{align*}
where the latter equality follows from unconditionality of $f$. In particular, as $x_1 \in [-y_1,y_1]$, we have for any $(x_2,...,x_d)\in \RR^{d-1}$,
\[
 f(x_1,x_2,...,x_d) \geq  f(y_1,x_2,...,x_d).
\]
This gives, for any $\lambda \in [0,1]$,
\begin{align*}
    f_{k}(\lambda x_{1}+(1-\lambda)y_1,x_2,\dots x_{k})
    & = \int f(\lambda x_{1}+(1-\lambda)y_1,x_2,\dots x_{d})d\mu_{k+1}\dots d\mu_{d}\\
    & \geq  \int f(y_1,x_2,\dots x_{d})d\mu_{k+1}\dots d\mu_{d}\\
    & = \min \left\{ f_{k}(x_{1},x_2,\dots x_{k}), f_{k}(y_1,x_2,\dots x_{k}) \right\}.
\end{align*}
By symmetry, the case $|y_1|\leq|x_1|$ follows, which finishes the proof.
\end{proof}

\section{A global approach for product measures}\label{sec:global}

We provide in this section another proof of Theorem \ref{thm:Hu-produit-A} and we provide the proof of  Theorem \ref{thm:Harge-Royen-global}.
The first  main ingredient that will be used is a generalization of Hoeffding's covariance identity  \eqref{eq:cov-k} for product measures. 
 The two other ingredients are a generalization to product measures of the  Hoeffding's covariance relations of Lemmas \ref{thm:cov2} and \ref{thm:cov-mod-a} and the use of FKG inequalities.

\subsection{Duplication and a generalization of Hoeffding's covariance identity}
We first present in Lemma \ref{lem:var-prod} a duplication argument  for the  covariance of a product measure.  Similar duplication representations are well known, see e.g. \cite{chatterjee:stein}.  
 We then deduce in Proposition \ref{prop:var-prod} a  generalization of Hoeffding's covariance identity for product measures.

 \begin{lem}\label{lem:var-prod}\
Let $\mu=\mu_1 \otimes \cdots \otimes \mu_d$ be a product measure on $\R^d$.
Under suitable integrable conditions one has 

\begin{equation}\label{eq:cov-prod}
{\rm Cov}_\mu(f,g)= \frac{1}{2} \sum_{i=1}^d \E [\Delta_i f(X,X') \tilde \Delta_i g(X,X')],
\end{equation}
where $X$ and $X'$ are two independent random variables of law $\mu$,
\[
\Delta_i f(X,X') = f(X_1,\dots,X_i , \dots , X_d) - f(X_1,\dots,X_i', \dots , X_d)
\]
and 
\[
\tilde \Delta_i g(X,X') = g(X_1,\dots, X_i ,X_{i+1}' \dots , X_d') - g(X_1,\dots,X_i', X_{i+1}' \dots , X_d').
\]
\end{lem}

\begin{proof}[Proof of Lemma \ref{lem:var-prod}] 
Let $X'$ be an independent copy of $X$ with law $\mu$. By symmetrization and then the use of a telescopic sum, one has
\begin{eqnarray*}
\cov_\mu(f,g)&=& \E[f(X) (g(X) -g(X')) ]\\
               &= & \sum_{i=1}^d  \E[f(X) \tilde \Delta_i g(X,X')  ]\\
               &=& \sum_{i=1}^d  \E\left[ U_i(X,X') \right], 
\end{eqnarray*}
where we   define $U_i(X,X')= f(X) \tilde \Delta_i g(X,X')$.  Let us  denote $(X,X')^{\{j\}}$ to be the random vector  given by
\[
(X,X')^{\{j\}}= \left(( X_1, \dots,X_{j-1}, X_j', X_{j+1}, \dots, X_d), ( X_1, \dots,X_{j-1}, X_j, X_{j+1}', \dots, X_d')\right)
\]
 We also write $(X,X')^{\{j\}}= \left(X^{\{j\}},X'^{\{j\}} \right)$ with the slight abuse of notation that $ X^{\{j\}}$ depends on $(X,X')$.           
Since $\mu$ is a product measure, for each $i$, $(X,X')^{\{i\}}$ is also of law $\mu\otimes \mu$ and thus
\begin{eqnarray*}
\E\left[ U_i  (X,X') \right] &=& \E\left[ U_i \left((X,X')^{\{i\}}\right) \right]\\
               &= &  - \E[f(X^{\{i\}}) \tilde \Delta_i g (X,X')  ]\\
\end{eqnarray*}
since $\tilde \Delta_i g \left( (X,X')^{\{i\}}\right) = -\tilde \Delta_i g (X,X')$ and thus
\[
\E\left[ U_i  (X,X') \right] = \frac{1}{2} \E\left[ U_i  (X,X') \right] + \frac{1}{2}  \E\left[ U_i \left((X,X')^{\{i\}}\right) \right]
= \frac{1}{2}\E [\Delta_i f(X,X') \tilde \Delta_i g(X,X')]
\]
and the result follows.
\end{proof}

From the duplication argument, one obtains the following generalization to product measures of Hoeffding's covariance identity.
 \begin{prop} \label{prop:var-prod}
 Let $\mu=\mu_1\otimes \dots \otimes \mu_d$ be a product measure on $\R^d$. 
Let $f,g:\R^d \to \R$ be some coordinatewise absolutely continuous functions in $L^2(\mu)$, then
\begin{equation}\label{eq:var-prod}
 \cov_\mu(f,g)
= \sum_{i=1}^d  \iint_{
x,x' \in  \R^d}  
\partial_i f(x) 
k_{\mu_i}(x_i,x_i')  
\partial_i g(\underline{x}_{i-1},\overline{x'}_{i}) 
dx_i dx_i'   d\mu(x_{-i}) d\mu(x'_{-i})
\end{equation}
where for $x_i, x_i'\in \R$, $k_{\mu_i}$is the standard Hoeffding kernel for the marginal $\mu_i$:
\[
k_{\mu_i}(x_i,x_i')= F_{\mu_i} (x_i\wedge x_i') -F_{\mu_i}(x_i) F_{\mu_i}(x_i')
\]
and for $x,x'\in \R^d$,
$(\underline{x}_{i-1},\overline{x'}_{i})= (x_1,\dots,{x}_{i-1}, x_i', \dots,{x}'_{d})$,
$x_{-i}=(x_1,\dots,x_{i-1}, x_{i+1},\dots x_d)$ and $d\mu(x_{-i})= d\mu_1(x_1) \dots d \mu_{i-1}(x_{i-1}) d \mu_{i+1}(x_{i+1})\dots  d \mu_{d}(x_{d})$.

 \end{prop}
\begin{proof}[Proof of Proposition \ref{prop:var-prod}]
We consider one term in the sum of the covariance formula of Lemma \ref{lem:var-prod}. We have
\begin{align*}
& \E [\Delta_i f(X,X') \tilde \Delta_i g(X,X')]\\
=
  &  \iint_{x,x' \in  \R^d} 
\begin{pmatrix} f(\underline{x}_{i-1}, x_i, \overline{x}_{i+1})\\
                   - f ( \underline{x}_{i-1},x_i', \overline{x}_{i+1}) 
                 \end{pmatrix}
                \begin{pmatrix} g(\underline{x}_{i-1}, x_i, \overline{x'}_{i+1})\\
                   - g( \underline{x}_{i-1},x_i', \overline{x'}_{i+1}) 
                 \end{pmatrix}  d\mu(x)d\mu(x')\\
=&  \iint_{x,x' \in  \R^d}  \iint_{s_i,t_i\in \R} 
 \partial_if(\underline{x}_{i-1}, s_i, \overline{x}_{i+1})  \partial_i g(\underline{x}_{i-1}, t_i, \overline{x'}_{i+1}) \\
 & \hspace{3cm} \left( \bone_{\{s_i\leq x_i\}} -  \bone_{\{s_i\leq x_i'\}}  \right) \left( \bone_{\{t_i\leq x_i\}} -  \bone_{\{t_i\leq x_i'\}} \right)  ds_i dt_i d\mu(x)d\mu(x').
\end{align*}
Furthermore,
\begin{eqnarray*}
 & &\iint_{x_i,x_i' \in \R}  \left( \bone_{\{s_i\leq x_i\}} -  \bone_{\{s_i\leq x_i'\}}  \right) \left( \bone_{\{t_i\leq x_i\}} -  \bone_{\{t_i\leq x_i'\}} \right) d\mu_i(x_i) d\mu_i(x_i')\\
&= & 2 \left( \P(X_i\geq \max(s_i,t_i)) - \P(X_i\geq s_i) \P(X_i\geq t_i )\right) \\
&= & 2 \left(F_{\mu_i} (s_i\wedge t_i) -F_{\mu_i}(s_i) F_{\mu_i}(t_i)\right)\\
&= & 2 k_{\mu_i}(s_i,t_i)
\end{eqnarray*}
 and the proof follows by Fubini theorem and by a change in the name of the letters in the integral.
\end{proof}


We now study some symmetry properties of this covariance representation.

\begin{lem}\label{lem:kmu-sym}
Assume $\mu_i$ is a symmetric one dimensional measure, then the kernel $k_{\mu_i}$ is even, that is 
\[k_{\mu_i}(-s_i,-t_i) = k_{\mu_i} (s_i,t_i).\]
\end{lem}

\begin{proof}
Without loss of generality assume that $s\leq t$,  then $-t\leq -s$, and 
\begin{eqnarray*}
k_{\mu_i} (-s,-t) &=& F_{\mu_i} (-t) - F_{\mu_i} (-s) F_{\mu_i} (-t) \\
                    &=& (1-F_{\mu_i} (t)) - (1- F_{\mu_i} (s)) (1 - F_{\mu_i} (t))\\
                    &= & F_{\mu_i} (s) - F_{\mu_i} (s) F_{\mu_i} (t)\\
                    &=& k_{\mu_i} (s,t).
\end{eqnarray*}
\end{proof}
As a consequence, one obtains the following result.
\begin{lem} \label{lem:even}
Assume that $\mu$ is a symmetric product measure on $\R^d$. 
Let $f,g:\R^n \to \R$ be two even functions. Then, for any $1\leq i \leq d$,
\[
\iint_{
x,x' \in  \R^d}  
\partial_i f(x) 
k_{\mu_i}(x_i,x_i')  
 g(\underline{x}_{i-1},\overline{x'}_{i}) 
dx_i dx_i'   d\mu(x_{-i}) d\mu(x'_{-i})=0.
\]
\end{lem}
\begin{proof}
The result follows from using the change of variables $(a,b)=(-x,-x')$ on $\R^{2d}$ and the fact that $\partial_i f $ is odd, $g$ is even and that the kernel $k_{\mu_i}$ is even.
\end{proof}
We also derive the following formulas, that will be instrumental in our proofs.
\begin{lem}\label{lem:cov-prod-ei}
Assume $\mu$ is a  product measure on $\R^d$. For each $1\leq k \leq d$, let $a_k(x_k)$ be a positive function on $\R$ and let $A_k$ be a primitive, centered with respect to $\mu_k$.
Let $f:\R^n \to \R$ be a coordinatewise absolutely continuous function. Then for  any $1\leq i\leq d $, one has
\[
\iint_{
x,x' \in  \R^d}
\partial_i f(x)  
k_{\mu_i}(x_i,x_i')  a_i(x_i')dx_i dx_i'   d\mu(x_{-i}) d\mu(x'_{-i})= \cov_\mu(f,A_i(x_i)).
\]
In particular,
\[
\iint_{
x,x' \in  \R^d}
\partial_i f(x)  
k_{\mu_i}(x_i,x_i')  dx_i dx_i'   d\mu(x_{-i}) d\mu(x'_{-i})= \cov_\mu(f,x_i),
\]
where, by a slight abuse of notation, $x_i$ stands for the $i$th-coordinate function. It also holds 
\[
\iint_{
x,x' \in  \R^d}
k_{\mu_i}(x_i,x_i') dx_i dx_i'   d\mu(x_{-i}) d\mu(x'_{-i})=  \Var_\mu(x_i)=\Var(\mu_i).
\]
\end{lem}
\begin{proof}
The proof is a direct application of Proposition \ref{prop:var-prod} with   $g(x)=A_i(x_i)$, noticing that only one term in the sum is different from zero. 
\end{proof}

\subsection{Hoeffding's covariance relation for product measures}
The main result here is Lemma \ref{lem:cov2-product} where  a similar relation as in Lemma \ref{thm:cov2} is given for product measures.

Let $\mu=\mu_{1}\otimes\dots\otimes\mu_{d}$ be a product measure
and write $\Gamma=\Gamma_{\mu}$ its covariance matrix.  Since $\mu$ is a product measure, it is diagonal with $\Gamma_{i,i}=\Var_{\mu}(x_{i})=\cov_{\mu}(x_{i},x_{i})$. 

Since the kernels $k_{\mu_i}$ are non-negative, one can introduce the probability measures on $\R^{2d}$, defined for $1\leq i \leq d$ by

\[
d\mu_{(i)}^{(1)}(x,y)=\frac{1}{\Gamma_{i,i}} k_{\mu_i}(x_i,y_i) dx_i dy_i   d\mu(x_{-i}) d\mu(y_{-i}).
\]
If $f$ and $g$ are positive and integrable, we also introduce the following probability measures,
\[
d\mu_{(i),f}^{(1)}(x,x')=\frac{1}{Z_{i,f}} f(x) k_{\mu_i}(x_i,x_i') dx_i dx'_i   d\mu(x_{-i}) d\mu(x'_{-i}),
\]
with 
\[
Z_{i,f}= \iint_{x,x'} f(x)k_{\mu_i}(x_i,x'_i) dx_i dx'_i   d\mu(x_{-i}) d\mu(x'_{-i}) 
\]
and 
\[
d\mu_{(i),f,g}^{(1)}(x,x')=\frac{ 1}{Z_{i,f,g}} f(x)    
k_{\mu_i}(x_i,y_i)g( \underline{x}_{i-1}, \overline{x'}_{i}) dx_i dx'_i   d\mu(x_{-i}) d\mu(x'_{-i}),
\]
with 
\[
Z_{i,f,g}= \iint_{x,x'}  f(x)   k_{\mu_i}(x_i,x'_i) g( \underline{x}_{i-1}, \overline{x'}_{i}) dx_i dx'_i   d\mu(x_{-i}) d\mu(x'_{-i}).
\]
The quantity $Z_{i,f}$ can still be written as a covariance with respect to $\mu$:
$Z_{i,f}= \cov_\mu(F_i(x), x_i)$ where $F_i$ is a function such that $\partial_i F_i(x)= f(x)$. This is not anymore the case for $Z_{i,f,g}$.

In the case of a product measure $\mu$, Lemma \ref{thm:cov2}
generalizes as follows.

\begin{lem}\label{lem:cov2-product} Let $f,g:\R^{d}\to\R$ be in $L^2(\mu)$ and coordinatewise absolutely continuous.
\begin{enumerate}
    \item 
Then, 
\begin{align*}
 \cov_\mu(f,g)    
= &\sum_{i=1}^d  {\Gamma_{i,i}} \cov_{\mu_{(i)}^{(1)}} (\partial_i f(x) ,\partial_i g( \underline{x}_{i-1}, \overline{x'}_{i}))\\
 & +\sum_{i=1}^d
 \frac{1}{\Gamma_{i,i}}
 \cov_\mu(f(x),x_i) \cov_\mu(g(x),x_i).\\
\end{align*}
\item If moreover $f=e^{-\phi}$, then 
\begin{align*}
\Cov_{\mu}(f(x),g(x))  = &\sum_{i=1}^d Z_{i,f}  \Cov_{\mu_{(i),f}^{(1)}}(-\partial_{i} \phi(x),\partial_i g(\underline{x}_{i-1},\overline{x'}_{i}) )\\
&+\sum_{i=1}^d \Cov_{\mu}(f(x),x_{i})\\
 & \hspace{1.6cm} \times \left(\iint f(x) \frac{k_{\mu_i}^{(1)}(x_i,x_i')}{\Gamma_{i,i}} \partial_{i}g (\underline{x}_{i-1},\overline{x'}_{i})) dx_i dx_i'   d\mu(x_{-i}) d\mu(x'_{-i})\right).
\end{align*}
In particular, if  moreover $f$ is orthogonal to  the linear functions $x_i$, $1\leq i\leq d$,
\[
\Cov_{\mu}(f(x),g(x))  = \sum_{i=1}^d Z_{i,f}  \Cov_{\mu_{(i),f}^{(1)}}(-\partial_{i} \phi(x),\partial_i g(\underline{x}_{i-1},\overline{x'}_{i}) ).
\]

\item If $f=e^{-\phi}$ and  $g=e^{-\psi}$,

\begin{align*}
& \Cov_{\mu}(f(x),g(x)) \\
= & \sum_{i=1}^d Z_{i,f,g}  \Cov_{\mu_{(i),f,g}^{(1)}}(\partial_{i} \phi(x),\partial_i \psi(\underline{x}_{i-1},\overline{x'}_{i}) )\\
&+ \sum_{i=1}^d Z_{i,f,g} 
\left(\iint\partial_{i}f(x)\frac{k_{\mu_i}^{(1)}(x_i,x_i')}{Z_{i,f,g} }  g (\underline{x}_{i-1},\overline{x'}_{i})) dx_i dx_i'   d\mu(x_{-i}) d\mu(x'_{-i}) \right)\\
& \hspace{1.4cm}\times 
\left(\iint f(x) \frac{k_{\mu_i}^{(1)}(x_i,x_i')}{Z_{i,f,g}} \partial_{i}g (\underline{x}_{i-1},\overline{x'}_{i})) dx_i dx_i'   d\mu(x_{-i}) d\mu(x'_{-i})\right).\\
\end{align*}
In particular, if the measure $\mu$ is symmetric and if both $f$ and $g$ are even, then 
\[
\Cov_{\mu}(f(x),g(x)) =  \sum_{i=1}^d Z_{i,f,g}  \Cov_{\mu_{(i),f,g}^{(1)}}(\partial_{i} \phi(x),\partial_i \psi(\underline{x}_{i-1},\overline{x'}_{i}) ).
\]
\end{enumerate}
\end{lem}

In fact, we shall use in the sequel the following slight weighted generalization, similar to the one of  Lemma \ref{thm:cov-mod-a}.

\begin{lem}\label{lem:cov2-product-mod}
Let $f,g:\R^{d}\to\R$ be in $L^2(\mu)$ and coordinatewise absolutely continuous.
\begin{enumerate}
    \item 
Then,  
\begin{align*}
 \cov_\mu(f,g)    
= &\sum_{i=1}^d  {\var_{\mu_i}(A_i)} \cov_{\mu_{(i),a_i,a_i}^{(1)}} (\partial_i f(x) ,\partial_i g( \underline{x}_{i-1}, \overline{x'}_{i}))\\
 & +\sum_{i=1}^d
 \frac{1}{\var_{\mu_i}(A_i)}
 \cov_\mu(f(x),A_i(x_i)) \cov_\mu(g(x),A_i(x_i)).\\
\end{align*}
\item If moreover $f=e^{-\phi}$ and if $f$ is orthogonal to  the  functions $A_i(x_i)$, $1\leq i\leq d$, then 
\[
\Cov_{\mu}(f(x),g(x))  = \sum_{i=1}^d Z_{i,a_i f,a_i}  \Cov_{\mu_{(i),a_i f,a_i}^{(1)}}(-\partial_{i} \phi(x),\partial_i g(\underline{x}_{i-1},\overline{x'}_{i}) ).
\]
\item If  moreover $f=e^{-\phi}$ and  $g=e^{-\psi}$ and if  the measure $\mu$ is symmetric, the function $a_k$ are even and   both $f$ and $g$ are even, then 
\[
\Cov_{\mu}(f(x),g(x)) =  \sum_{i=1}^d Z_{i,a_i f,a_i g}  \Cov_{\mu_{(i),a_i f,a_ig}^{(1)}}(\partial_{i} \phi(x),\partial_i \psi(\underline{x}_{i-1},\overline{x'}_{i}) ).
\]
\end{enumerate}
\end{lem}

Since the other points are somehow similar, we only do the proof for the first item of Lemma \ref{lem:cov2-product}.
\begin{proof}[Proof for the first item of Lemma \ref{lem:cov2-product}] 
 From Proposition \ref{prop:var-prod} and Lemma \ref{lem:cov-prod-ei}, one has
\begin{align*}
 & \Cov_{\mu}(f(x),g(x))\\
= &\sum_{i=1}^d  \iint_{
x,x' \in  \R^d}  
\partial_i f(x) 
k_{\mu_i}(x_i,x_i')  
\partial_i g(\underline{x}_{i-1},\overline{x'}_{i}) 
dx_i dx_i'   d\mu(x_{-i}) d\mu(x'_{-i})\\
= & \sum_{i}\Gamma_{i,i} \Cov_{\mu_{(i)}^{(1)}}(\partial_{i}f(x),\partial_i g(\underline{x}_{i-1},\overline{x'}_{i}) )\\
&+ \sum_{i}\Gamma_{i,i}
\left(\iint\partial_{i}f(x)\frac{k_{\mu_i}^{(1)}(x_i,x_i')}{\Gamma_{i,i}} dx_i dx_i'   d\mu(x_{-i}) d\mu(x'_{-i}) \right)\\
& \hspace{1cm} \times \left(\iint\partial_{i}g (\underline{x}_{i-1},\overline{x'}_{i}) \frac{k_{\mu_i}^{(1)}(x_i,x_i')}{\Gamma_{i,i}} dx_i dx_i'   d\mu(x_{-i}) d\mu(x'_{-i})\right)\\
= & \sum_{i=1}^{d}\Gamma_{i,i}\;\Cov_{\mu_{(i)}^{(1)}}(\partial_{i}f(x),\partial_{i}g(y))+\sum_{i=1}^{d}\frac{1}{\Gamma_{i,i}}\;\Cov_{\mu}(f(x),x_{i})\,\Cov_{\mu}(g(x),x_{i}).
\end{align*}
\end{proof}

\subsection{Another proof  of Theorem \ref{thm:Hu-produit-A} and a proof of Theorem \ref{thm:Harge-Royen-global}}
Before we turn to the announced proofs, we highlight with the next statement that under our assumptions, the new probability measures on $\R^{2d}$ satisfy the Holley condition and thus the FKG inequality.

Recall that $\mu= \mu_1 \otimes \dots \otimes \mu_d$ is a product measure with marginals $\mu_k$, $k=1,\dots,d$, admitting densities, denoted by $\exp(-V_k)$, with respect to the Lebesgue measure. For some index $i\in \{1,\dots,d\}$ and for $f$ and $g$ some positive functions on $\RR^d$, the kernel $k_{(i),f,g}$ is defined on $\RR^{2d}$ by
\[
k_{(i),f,g}=f(x)k_{\mu_i}(x_i,y_i)g(x)\prod_{j\neq i}e^{-V_j(x_j)}\prod_{j\neq i}e^{-V_j(x'_j)}.
\]
The measure $\mu_{(i),f,g}^{(1)}$ has a density on $\RR^{2d}$ equal to 
\[
d\mu_{(i),f,g}^{(1)}(x,x')=\frac{ 1}{Z_{i,f,g}} k_{(i),f,g}^{(1)} dx dx',
\]
with 
\[
Z_{i,f,g}= \iint_{x,x'}  f(x)   k_{\mu_i}(x_i,x'_i) g( \underline{x}_{i-1}, \overline{x'}_{i}) dx_i dx'_i   d\mu(x_{-i}) d\mu(x'_{-i}).
\]
\begin{prop}\label{prop:fkg-FKG}
Let $\mu= \mu_1 \otimes \dots \otimes \mu_d$ be a product measure on $\R^d$ and grant the above notations. One has
\begin{enumerate}
\item For all  $1\leq i \leq d$, the measures $\mu_{(i)}^{(1)}$ and $\mu_{(i),a_i,a_i}^{(1)}$ satisfy the Holley condition \eqref{eq:k-FKG}. 
Moreover, for any choice of signs $(\ve_1, \dots, \ve_d)\in \{ +1,-1\} ^d$, the kernels   $\tilde k_{(i)}$ and $\tilde k_{(i),a_i,a_i}$ defined by 
\[
\tilde k_{(i)} (x,x')= k_{(i)}^{(1)} (\ve x, \ve x') \textrm{ and }  \tilde k_{(i),a_i,a_i} (x,x')= k_{(i),a_i,a_i} (\ve x, \ve x')
\]
with  $k_{(i),a_i,a_i} $ the density - up to the constant factor $\var_{\mu_i}(A_i)$ - of the measure  $\mu_{(i),a_i,a_i}^{(1)}$ with respect to the Lebesgue measure on $\R^{2d}$ and   
\[
(\ve x, \ve x')= (\ve_1 x_1, \dots, \ve_d x_d, \ve_1 x'_1, \dots, \ve_d x'_d),
\]
satisfy the Holley condition \eqref{eq:k-FKG}. 
\item Assume that $f=e^{-\phi}$ and  that for all $1\leq i,j \leq d$ with $i\neq j$,
\[
\partial_{i,j} \phi (x) \leq 0
\]
 then for all $1\leq i \leq d$, the measures $\mu_{(i),f}^{(1)}$ and $\mu_{(i),a_if,a_i}^{(1)}$ satisfy the Holley condition \eqref{eq:k-FKG}. 

\item Assume that   $f=e^{-\phi}$ and $g=e^{-\psi}$
and  that for all $1\leq i,j \leq d$ with $i\neq j$,
\[
\partial_{i,j} \phi (x) \leq 0 \textrm{ and }\partial_{i,j} \psi (x) \leq 0
\]
then for all  $1\leq i \leq d$, the measures $\mu_{(i),f,g}^{(1)}$ and $\mu_{(i),a_i f ,a_i g}^{(1)}$ satisfiy the Holley condition \eqref{eq:k-FKG}.
\end{enumerate}
\end{prop}
Note that in the latter proposition, the signs of the second-order cross derivatives for $\phi$ and $\psi$ should be both non-positive.

\begin{proof} 
The  logarithm  $H_{(i),a_i,a_i}^{(1)}$ of the density  of $\mu_{(i),a_i,a_i}^{(1)}$ with respect to the Lebesgue measure on $\R^{2d}$ is given by
\[
H_{(i),a_i,a_i}^{(1)}(x,x')=  \ln k_{\mu_i}(x_i,x_i') + \ln  a_i(x_i)+ \ln a_i (x_i')  - \sum_{j\neq i} V_j(x_j) - \sum_{j\neq i} V_j(x_j').
\]
Since $k_{\mu_i}$  is a totally positive kernel on $\R^2$, it follows easily that $H_{(i),a_i,a_i}^{(1)}$ satisfies \eqref{eq:cond-Holley-H}.
 Now for $(\ve_1, \dots, \ve_d)\in \{ +1,-1\} ^d$ fixed, the logarithm $\tilde H_{(i),a_i,a_i}^{(1)}$ of the kernel $\tilde k_{(i),a_i,a_i}^{(1)}$ is given by:
 \[
\tilde H_{(i),a_i,a_i}^{(1)}(x,x')=  \ln k_{\mu_i}(\ve_i x_i,\ve_i x_i') + \ln  a_i(\ve_i x_i)+ \ln a_i (\ve_i x_i')  - \sum_{j\neq i} V_j(\ve_j x_j) - \sum_{j\neq i} V_j(\ve_j x_j').
\] 
Since the kernel $k_{\mu_i}(\ve_i x_i,\ve_i x_i')$  is still totally positive on $\R^2$,
the proof of the first point follows.
We turn to the proof of the second point.
The  logarithm  $H_{(i),a_i f, a_i}^{(1)}$ of the density  of $\mu_{(i),a_i f,a_i}^{(1)}$ with respect to the Lebesgue measure on $\R^{2d}$ satisfies 
\[H_{(i),a_if, a_i}^{(1)}(x,x')= -\phi(x) + H_{(i),a_i , a_i}^{(1)}. \]
From assumption \eqref{eq:cond-l2-phi-g} and  Remark \ref{rmk:bakry-michel}, the function $x\to -\phi(x)$ satisfies \eqref{eq:cond-Holley-H} on $\R^d$ and thus clearly the function $(x,x')\to -\phi(x)$ also satisfies \eqref{eq:cond-Holley-H} on $\R^{2d}$.
Finally, by summation, Inequality \eqref{eq:cond-Holley-H} is also valid on $\R^{2d}$ for $H_{(i),a_if,a_i}^{(1)}$.  
The proof for the third point is similar and we omit the details.
\end{proof}

We now provide another proof of Theorem \ref{thm:Hu-produit-A}.
\begin{proof}[Another proof of Theorem \ref{thm:Hu-produit-A}]
 Let $f$ and $g$ be two functions on $\R^d$ satisfying \eqref{eq:cond-l2-fg-mod}. 
 We first assume that all the signs of the second derivatives in Assumption \ref{eq:cond-l2-fg-mod} are non-negative. 
 By Lemma \ref{lem:cov2-product-mod}, one has 
 \begin{align*}
 & \cov_\mu(f,g)    - \sum_{i=1}^d 
 \frac{1}{\var_{\mu_i}(A_i)}
 \cov_\mu(f(x),A_i(x_i)) \cov_\mu(g(x),A_i(x_i))\\
=& \sum_{i=1}^d  {\var_{\mu_i}(A_i)} \cov_{\mu_{(i),a_i,a_i}^{(1)}} \left( \frac{\partial_i f(x)}{a_i(x_i)} , \frac{\partial_i g( \underline{x}_{i-1}, \overline{x'}_{i})}{a_i(x_i')} \right).\\
\end{align*}
Furthermore, by Proposition \ref{prop:fkg-FKG}(1), the measure $\mu_{(i),a_i,a_i}^{(1)}$, for $i\in \{1,\dots,d\}$, satisfies the Holley condition \eqref{eq:k-FKG}. By condition \eqref{eq:cond-l2-fg} both functions $(x,x')\to\partial_i f(x)/a_i(x_i)$ and $(x,x')\to\partial_i g( \underline{x}_{i-1}, \overline{x'}_{i})/a_i(x_i')$ are coordinate increasing on $\R^{2d}$ and thus, for each  $1\leq i\leq d$,
\[
 \cov_{\mu_{(i),a_i,a_i}^{(1)}} \left( \frac{\partial_i f(x)}{a_i(x_i)} , \frac{\partial_i g( \underline{x}_{i-1}, \overline{x'}_{i})}{a_i(x_i')} \right) \geq 0.
\]
Summing these  inequalities ends the proof in this specific case. 
In the general case, for any $\ve= (\ve_1, \dots, \ve_d)\in \{ +1,-1\} ^d$, by the change of variable $(\tilde x, \tilde x')= (\ve x,\ve x')$, one has
\[
 \cov_{\mu_{(i),a_i,a_i}^{(1)}} \left( \frac{\partial_i f(x)}{a_i(x_i)} , \frac{\partial_i g( \underline{x}_{i-1}, \overline{x'}_{i})}{a_i(x_i')} \right) = \cov_{\tilde \mu_{(i),a_i,a_i}^{(1)}} \left( \frac{\partial_i f( \ve x)}{a_i(\ve_i x_i)} , \frac{\partial_i g(  \underline{ \ve x}_{i-1}, \overline{ \ve x'}_{i})}{a_i( \ve_i x_i')} \right)
\]
and for each $1\leq i \leq d$, it is possible to find some vector $\ve= (\ve_1, \dots, \ve_d)\in \{ +1,-1\} ^d$ such that  $\frac{\partial_i f( \ve x)}{a_i(\ve_i x_i)}$  and  $ \frac{\partial_i g(  \underline{ \ve x}_{i-1}, \overline{ \ve x'}_{i})}{a_i( \ve_i x_i')} $ are both coordinate increasing. More precisely, it suffices to take $\ve_j=\textrm{sign } \partial_j \left( \frac{\partial_i f}{a_i} \right)$. By Lemma \ref{lem:cov2-product-mod}(1), the measures $\tilde \mu_{(i),a_i,a_i}^{(1)}$ also satisfy  the Holley condition and the result follows from the FKG inequality.
\end{proof}

We turn now to the proof of Theorem \ref{cor:Harge-Royen-global}, where we add some symmetries.


\begin{proof}[Proof of Theorem \ref{cor:Harge-Royen-global}]
 Let $f=e^{-\phi}$ and $g$ be two functions on $\R^d$ satisfying \eqref{eq:cond-ii-phi-g} and  \eqref{eq:cond-ij-phi-g} and assume that $f$ is orthogonal to the  functions $A_i$,  $1\leq i\leq d$.  By Lemma \ref{lem:cov2-product-mod}(2), 
 one has
\[
\Cov_{\mu}(f(x),g(x))  = \sum_{i=1}^d Z_{i,a_i f,a_i}  \Cov_{\mu_{(i),a_i f,a_i}^{(1)}}\left( -\frac{\partial_{i} \phi(x)}{a_i(x_i)},\frac{\partial_i g(\underline{x}_{i-1},\overline{x'}_{i})}{a_i(x_i')}\right) .
\]
Now for each $i$,  since $\phi$ satisfies \eqref{eq:cond-ij-phi-g}, by Proposition \ref{prop:fkg-FKG}(2)  the measure $\mu_{(i),a_if,a_i}^{(1)}$ satisfies the Holley condition.  Moreover  adding condition \eqref{eq:cond-l2-phi-g}  both functions $(x,x')\to \frac{\partial_{i} \phi(x)}{a_i(x_i)} $ and $(x,x')\to \frac{\partial_i g( \underline{x}_{i-1}, \overline{x'}_{i})}{a_i(x_i')}$ are  coordinate  increasing on $\R^{2d}$, and thus by the FKG inequality, for each  $1\leq i\leq d$, one has:
\[
\Cov_{\mu_{(i),a_i f,a_i}^{(1)}}\left( -\frac{\partial_{i} \phi(x)}{a_i(x_i)},\frac{\partial_i g(\underline{x}_{i-1},\overline{x'}_{i})}{a_i(x_i')}\right) \geq 0.
\]
 Theorem \ref{cor:Harge-Royen-global}(1) thus follows.
The proof of Theorem \ref{cor:Harge-Royen-global}(2)  is similar, since by Lemma \ref{lem:cov2-product-mod}(3), the symmetry assumptions made on $f=e^{-\phi}$ and $g=e^{-\psi}$, give that 
\[
\Cov_{\mu}(f(x),g(x))  = \sum_{i=1}^d Z_{i,a_i f,a_ig }  \Cov_{\mu_{(i),a_i f,a_i}^{(1)}}\left( \frac{\partial_{i} \phi(x)}{a_i(x_i)},\frac{\partial_i \psi(\underline{x}_{i-1},\overline{x'}_{i})}{a_i(x_i')}\right) .
\]
Finally, the assumptions \eqref{eq:cond-ii-phi-g} and \eqref{eq:cond-ij-phi-g} ensure  that the measure  $\mu_{i,f,g}^{(1)}$  satisfies the Holley condition and that the functions in the covariance are coordinate increasing, which gives the result.
\end{proof}

\section{Comments on the standard semi-group interpolation}
\label{sec:comments}

In this section, we explain what can be done using a standard covariance representation obtained by interpolation with  the associated diffusion semi-group (see \eqref{eq:cov-rep-interpolation} below) and why we did not follow this natural approach, but rather used instead the covariance representation of Proposition \ref{prop:var-prod}.

We consider here a probability measure $\mu=e^{-V}dx$ with a smooth
potential $V$. One can associate to it a diffusion semi-group with generator $L$ defined
for $f$ smooth with compact support by 
\[Lf=\Delta f-\nabla V\cdot\nabla f.\]
This diffusion operator is symmetric with respect to $\mu$: for $f,g\in\mathcal{C}_{c}^{\infty}(\R^{d})$,
\[
\int fLgd\mu=\int Lfgd\mu=-\int\nabla f\cdot\nabla gd\mu.
\]
Under mild conditions on $V$, one we can define alternatively the semi-group
associated to $L$ by the spectral theorem and functional calculus,
or by a stochastic representation (see \cite{bakry-gentil-ledoux} for further details), 
\[
P_{t}f(x)=e^{tL}(f)(x)=\E[f(X_{t}^{x})]
\]
for some Markov diffusion process $(X_{t}^{x})_{t\geq 0}$.
We assume  moreover that  the operator  $ -\L+\Hess V$, with $\L={\rm diag}(L,\dots,L)$ acting on gradients, is invertible. Note that this holds under some strong convexity of the potential $V$.
In this situation,  for  $f,g:\R^{d}\to\R$ satisfying some integrability conditions on $f$ and $g$,
one has
\begin{equation}
\cov_{\mu}(f,g)=\int_{\R^{d}}\nabla f(x)\cdot(-\L+\Hess V)^{-1}\nabla g(x)d\mu(x)\label{eq:cov-L+V}
\end{equation}
and thus
\[
\cov_{\mu}(f,g)=\iint_{\R^{d}\times\R^{d}}\nabla f(x)K(x,y)\nabla g(y)dxdy
\]
where $K$ is the matricial kernel (with respect to the Lebesgue measure)
of the operator $(-\L+\Hess V)^{-1}$.
Moreover, the matricial kernel $K(x,y)$ admits the following stochastic Feynman-Kac representation,
\begin{equation}
K(x,y)=e^{-V(x)}\int_{0}^{+\infty}\E[Y_{t,x}|X_{t}=y]p_{t}(x,y)e^{-V(y)}dy,\label{eq:K}
\end{equation}
where $p_{t}$ stands for the heat kernel associated to $P_{t}$ with respect
to the measure $\mu$ and $Y_{t,x}$ is the matrix satisfying the
following ordinary (random) differential equation, 
\begin{equation}
\frac{d}{dt}Y_{t,x}=- Y_{t,x} \Hess V(X_{t}^{x})\textrm{ for }t\geq0;\ Y_{0,x}=Id.\label{eq:FKY}
\end{equation}
In the case of  a product measure, we can write $V(x)=V_1(x_1)+\dots + V_d(x_d)$, for some real functions $V_k$. This gives the following generalization of Hoeffding's covariance identity,
\begin{equation}\label{eq:cov-rep-interpolation}
\Cov_{\mu}(f,g)=\sum_{i=1}^{d}\iint_{x,y\in\R^{d}}\partial_{i}f(x)\, \kappa_{i}(x,y)\,\partial_{i}g(y)dxdy,
\end{equation}
where for each $1\leq i\leq d$ , $k_{i}:\R^{2d}\to\R_+$ is the
kernel defined by 
\begin{equation}\label{eq:k1i}
\kappa_{i}(x,y)=\int_{t=0}^{\infty}\E\left[\exp\left(-\int_{0}^{t}V_{i}''(X_{s}^{x_{i},i})ds\right)|(X_{t}^{x_{i},i}=y_{i})\right]p_{t}(x,y) \, dt \, e^{-V(x)}e^{-V(y)}.
\end{equation}
This kernel also writes as
\[ 
\kappa_{i}(x,y)=\int_{t=0}^{\infty}  \kappa_{i,t} \, dt  
\]
with
\[ 
 \kappa_{i,t}(x,y):=   p_{t,i}^{V_i''}(x_i,y_i) \prod_{j=1, j\neq i}^{d}p_{t,j}(x_{j},y_{j})     e^{-V(x)}e^{-V(y)}, 
\]
where $p_{t,j}$ is the  kernel of the one dimensional diffusion semi-group with generator given by $L_j f(x_j):=f''(x_j) - V_j'(x_j) f'(x_j)$ and  where $p_{t,i}^{V_i''}$ is the kernel of the one dimensional Schrödinger semi-group, with generator given by $L_i^{V_i''} f(x_i):=L_if(x_i) + V_i''(x_i) f(x_i).$
In the case of the standard Gaussian measure, one has $p_{t,i}^{V_i''}=e^{-t}p_{t,i}$.

We highlight that we do not know whether, in dimension  $d\geq 2$,  the probability measure  with density proportional to $\kappa_{i}(x,y)$ satisfies the full FKG inequality on $\R^{2d}$.
But one can also notice, that due to the \emph{coincidence formula},  diffusion  kernels and Schrödinger kernels in
dimension one are totally positive (see Karlin \cite{karlin:book}).
As a consequence,  the kernels $\kappa_{i,t}$ satisfy the Holley condition.
And if slightly differently, one has
\[ 
\kappa_{i}(x,y)=\int_{t=0}^{\infty}  \kappa_{i,t} \, d\nu(t)   
\]
for some probability measure $\nu$ on $\R_+$, 
one can use the following  decomposition of the covariance, 
 \begin{equation}\label{eq:cov-mixture}
\cov_{\mu_{\kappa_i}}(u,v)=\int_{0}^{\infty}\cov_{\mu_{\kappa_{i,t}}}(u,v)d\nu(t)+\cov_{\nu}\left(t\to\int  u \, d\mu_{\kappa_{i,t}},t\to\int  v \,d\mu_{\kappa_{i,t}}\right)
\end{equation}
and apply it with $u(x,y):=\partial_i f(x) $ and $v(x,y):=\partial_i g(y)$.

 In view of proving Theorem \ref{thm:Hu-produit-A}, we were only able to pursue this approach when the marginals  of  $\mu_{\kappa_{i,t}}$  on $\R^d \times \R^d$ are constant for all $t>0$. In this situation,
 if moreover, $(x,y)\to u(x)$ and $(x,y)\to v(y)$,  the term related to $\cov_\nu$ appearing in the right-hand side of \ref{eq:cov-mixture} indeed vanishes and one obtains some \emph{partial FKG inequalities} for the measure $\mu_{\kappa_i}$ on $\R^{2d}$. 
Here, the terms ``partial'' means  that it is applied only to  coordinate increasing functions of the form $(x,y)\to u(x)$ and $(x,y)\to v(y)$.

This property that the marginals $\mu_{\kappa_{i,t}}$ are constant, holds for the standard Gaussian measure and this approach may be pursued, with a second order covariance representation, to recover
Theorem \ref{thm:harge}(1) for the standard Gaussian measure. Let us give some details. 

First,  the first order representation \eqref{eq:cov-rep-interpolation} is well known for the standard Gaussian measure (see \cite{bobkov-gotze-houdre}). The measures $\gamma_{\kappa_{i,t}}$ are in fact independent of $i$, they  are also Gaussian measures on $\R^{2d}$ and they have  fixed marginals on $\R^d \times \R^d$.
A second  order covariance representation for  $\gamma$ thus means   a first order covariance of the new measure(s) $\gamma_{\kappa_{i,t}}$ similar to \eqref{eq:cov-rep-interpolation}.
 It can be obtained either by a change of variable since  $\gamma_{\kappa_{i,t}}$ is still a  Gaussian measure or by solving explicitly the stochastic Feynman-Kac representation.
This method is similar to the one of \cite{hu-chaos}, except that the latter approach specifically uses the fact that the Gaussian measure is the density at time $1$ of  the classical heat semi-group, whereas instead we use here the Orstein-Uhlenbeck operator.

Finally, for general product measures, the ``constant marginal property'' also holds for the modified kernels $k_{\mu_i,(a_i,a_i)}$, with the  choice $a_i(x_i)=\frac{1}{g_i'(x_i)}$ where $g_i$ is (if it exists) the first non-trivial eigenfunction associated to $L$. This leads to Theorem \ref{thm:Hu-produit-A}, but only for this specific choice. More importantly, this constant marginal property is valid for product measures under the symmetry assumptions of Theorem \ref{thm:Harge-Royen-global}(2) and this route may also be taken to provide another proof 
Theorem \ref{thm:Harge-Royen-global}(2).

\section{Examples}\label{sec:examples}
In this final section, we provide a couple of examples of possible applications of our results. 

First let $\mu$ be a product measure on $\R^d$
and for $\beta >0$ and consider the free energy, also known in the optimization community as the ``soft max'' function, 
\[
F_\beta(x):= \frac{1}{\beta} \ln\left( \sum_{i=1}^d e^{\beta x_i} \right). 
\]
By setting $p_i:= \frac {e^{\beta x_i}} { \sum_j e^{\beta x_j} }$, it satisfies
\[
\partial_i F_\beta = p_i \geq 0,
\]
\[
\partial_{ii} F_\beta(x)= \beta p_i(1-p_i)\geq 0 , \quad \partial_{ij} F_\beta(x)= -\beta p_i p_j \leq 0,\quad i\neq j.
\]
Thus, for any $\alpha,\beta>0$, Corollary \ref{cor:hu-produit} gives
\begin{equation}\label{cov-soft-max}
\Cov_\mu(F_\alpha, F_\beta) \geq    \sum_{i=1}^d  \frac{1} {\Var(\mu_i)} \cov(F_\alpha(x),x_i) \cov(F_\beta(x),x_i).
\end{equation}
Note that when $\alpha=\beta$, inequality \eqref{cov-soft-max} turns to the following Bessel inequality, 
\[
    \Var_\mu(F_\beta) \geq \sum_{i=1}^d  \frac{1} {\Var(\mu_i)}  \cov(F_\beta(x),x_i)^2. 
\]

\

We now turn to a second example.
Let $\mu $ be a symmetric product measure on $\R^d$. Under some integrability condition, for $J\geq 0$, we consider the probability measure: 
\[
d\mu_J(x)= \frac{1}{Z_J} e^{J \sum_{i=1}^{d-1}  x_i x_{i +1} } d\mu(x), \quad Z_J= \int_{\R^d}e^{J \sum_{i=1}^{d-1}  x_i x_{i +1} } d\mu(x).
\]

Let $\theta=(\theta_1,\dots,\theta_d)\in\R^d$ with $\theta_i \geq 0$, $1\leq i \leq d$,
then by  Corollary \ref{cor:Harge-Royen-global}, one has
\begin{equation}
    \int_{\R^d} \langle x,\theta \rangle^2 d \mu_J(x) \geq  \int_{\R^d} \langle x,\theta \rangle^2 d \mu(x).
\end{equation}

\section{Appendix}
We consider here some product probability  measures on $\R^d$ whose marginals  are   mixtures of centered Gaussian variables. This class of probability measures was investigated in \cite{Eskenazis}, where the authors prove that they satisfy \eqref{eq-harge-gaussienne-niv3} and provide interesting examples. Here we show that those measures also satisfy \eqref{eq-harge-gaussienne-niv2}.

We consider Gaussian mixtures of the form  
 \begin{equation}\label{def-mixture}
\mu=\iint_{ \sigma\in (0,\infty)^d} \gamma_{\Gamma_\sigma} d\nu(\sigma),
 \end{equation}
where $\gamma_{\Gamma_\sigma}$ is  the centered Gaussian random vector in $\R^d$ with covariance matrix  $\Gamma_\sigma= diag(\sigma_1^2, \dots,\sigma_d^2)$ and where $\nu$ is also a product measure on $(0,\infty)^d$.
 
\begin{theo}\label{thm:mixture-Harge}
Let $\mu$ be  a product probability  measure on $\R^d$, whose marginals  are  mixture of centered Gaussian variables.
Then \eqref{eq-harge-gaussienne-niv2} holds.
\end{theo}

 The proof relies on the following Lemma, the key point of which being that no symmetry assumption is required in the convex situation.
\begin{lem}\label{lem:coord-increase} The following points hold.
\begin{enumerate}

\item Let  $g$ be a  convex function on $\R^d$, then the function
\[
(\sigma_1, \dots,\sigma_d)\in (0,\infty)^d \to \int g(y) d\gamma_{\Gamma_\sigma}(y)
\]
is  coordinatewise increasing on $(0,\infty)^d$.
\item Let  $f$ be a  quasi-concave  and even function on $\R^d$, then the function
\[
(\sigma_1, \dots,\sigma_d)\in (0,\infty)^d \to \int f(y) d\gamma_{\Gamma_\sigma}(y)
\]
is    coordinatewise decreasing  on $(0,\infty)^d$.
\end{enumerate}
\end{lem}

\begin{proof}[Proof of Theorem \ref{thm:mixture-Harge}]
Let $f$ be a log-concave and even function and let $g$ be a convex function.
Using the decomposition of the covariance  \eqref{eq:cov-mixture}, one has
\begin{align*}
\cov_\mu (f,g)&= \iint_{(0,\infty)^d} \cov_{\gamma_{\Gamma_\sigma}}(f,g) d\nu(\sigma) \\
 & \quad + \cov_\nu \left( \sigma\in (0,\infty)^d \to \int f d\gamma_{\Gamma_\sigma} , \sigma\in (0,\infty)^d \to \int g d\gamma_{\Gamma_\sigma}
\right).
\end{align*}
The rest of the proof consists in showing that the two terms in the right-hand side of the latter inequality are non-positive.
Firstly,  Hargé's result \eqref{eq-harge-gaussienne-niv2} also  applies to any (centered) Gaussian distribution (see \cite{harge}) and thus $\cov_{\gamma_{\Gamma_\sigma}}(f,g)\leq 0$. 
Secondly, we use Lemma \ref{lem:coord-increase}, since $f$ is log-concave and even, it is also quasi-concave and even, and thus the two functions 
\[
\sigma\in (0,\infty)^d \to \int g d\gamma_{\Gamma_\sigma} \textrm{ and }\sigma\in (0,\infty)^d \to \int f d\gamma_{\Gamma_\sigma}
\]
are respectively coordinatewise increasing and coordinatewise decreasing on $(0,\infty)^2$. The measure $\nu$ being a product measure, by the FKG inequality for product measure, the term  $\cov_\nu (\cdot,\cdot)$ is non-positive and the result follows.
\end{proof}
We turn now to the proof of Lemma \ref{lem:coord-increase}.
\begin{proof}[Proof of Lemma \ref{lem:coord-increase}]
Let $g$ be a convex function  on $\R^d$. By a change of variable, one directly has
\[
\int g(y) d\gamma_{\Gamma_\sigma}(y)= \int_{\R^d} g(\sigma_1 x_1, \dots,\sigma_d x_d)  d\gamma(x),
\]
where we recall that $\gamma$ is the standard Gaussian distribution. To  prove the desired  property, we compute for $1\leq l\leq d$,
\begin{align*}
\frac{\partial}{\partial \sigma_l}  \int_{\R^d} g(\sigma_1 x_1, \dots,\sigma_d x_d)  d\gamma(x)
&=  \int_{\R^d}  x_l \, \partial_l g(\sigma_1 x_1, \dots,\sigma_d x_d)  d\gamma(x)\\
& = \cov_{\gamma} ( x_l, \partial_l g(\sigma_1 x_1, \dots,\sigma_d x_d)).
\end{align*}
Furthermore, by the covariance representation \eqref{eq:cov-rep-interpolation} for the standard Gaussian measure, one has
\begin{align*}
\cov_{\gamma} ( x_l, \partial_l g(\sigma_1 x_1, \dots,\sigma_d x_d))
= \iint_{x,y\in \R^d}  \kappa(x,y) \sigma_l  \partial_{ll} g(\sigma_1 y_1, \dots,\sigma_d y_d) dx dy
\end{align*}
and this quantity is non-negative 
since $g$ is convex and  $\kappa(x,y)\geq 0$. The result follows.\\
For $f$ quasi-concave and even, we use the layer cake representation of $f$:
\[
    f(x)= \int_0^\infty \bone_{A_t}(x) dt \textrm{ and }  A_t:= \{x\in \R, f(x)\geq t \}.
\]
Here by assumption the $A_t$ are convex and even.
Since by Fubini, 
\[
\int_{\R^d} f(\sigma_1 x_1, \dots,\sigma_d x_d)  d\gamma(x)
= \int_0 ^\infty \int_{\R^d}  \mathbf 1_{A_t} (\sigma_1 x_1, \dots,\sigma_d x_d)  d\gamma(x) dt.
\]
the result follows  from \cite{Eskenazis} where the authors prove the following property: for each $t\geq 0$,  
\[
(\sigma_1, \dots,\sigma_d)\in (0,\infty)^d \to \int_{\R^d} \mathbf 1_{A_t} (\sigma_1 x_1, \dots,\sigma_d x_d)  d\gamma(x)
\]
is coordinatewise decreasing.
\end{proof}

\bibliographystyle{alpha}
\bibliography{BHS,chern}

\def\cprime{$'$} \def\cprime{$'$} \def\cprime{$'$} \def\cprime{$'$}
  \def\cprime{$'$} \def\cprime{$'$}
\begin{thebibliography}{CCEL13}

\bibitem[ABJ18]{abj}
M.~Arnaudon, M.~Bonnefont, and A.~Joulin.
\newblock Intertwinings and generalized {B}rascamp-{L}ieb inequalities.
\newblock {\em Rev. Mat. Iberoam.}, 34(3):1021--1054, 2018.

\bibitem[Bar19]{barthe:royen}
F.~Barthe.
\newblock L'in\'{e}galit\'{e} de corr\'{e}lation {G}aussienne [d'apr\`es
  {T}homas {R}oyen].
\newblock Number 407, pages Exp. No. 1124, 117--133. 2019.
\newblock S\'{e}minaire Bourbaki. Vol. 2016/2017. Expos\'{e}s 1120--1135.

\bibitem[BGH01]{bobkov-gotze-houdre}
S.~G. Bobkov, F.~G\"{o}tze, and C.~Houdr\'{e}.
\newblock On {G}aussian and {B}ernoulli covariance representations.
\newblock {\em Bernoulli}, 7(3):439--451, 2001.

\bibitem[BGL14]{bakry-gentil-ledoux}
D.~Bakry, I.~Gentil, and M.~Ledoux.
\newblock {\em Analysis and geometry of Markov diffusion operators}, volume 348
  of {\em Grundlehren der Mathematischen Wissenschaften [Fundamental Principles
  of Mathematical Sciences]}.
\newblock Springer, Berlin, 2014.

\bibitem[BM92]{bakry-michel}
D.~Bakry and D.~Michel.
\newblock Sur les in\'{e}galit\'{e}s {FKG}.
\newblock In {\em S\'{e}minaire de {P}robabilit\'{e}s, {XXVI}}, volume 1526 of
  {\em Lecture Notes in Math.}, pages 170--188. Springer, Berlin, 1992.

\bibitem[Bob96]{bobkov:extremal}
S.~Bobkov.
\newblock Extremal properties of half-spaces for log-concave distributions.
\newblock {\em Ann. Probab.}, 24(1):35--48, 1996.

\bibitem[CCEL13]{carlen-cordero-lieb}
E.~A. Carlen, D.~Cordero-Erausquin, and E.~H. Lieb.
\newblock Asymmetric covariance estimates of {B}rascamp-{L}ieb type and related
  inequalities for log-concave measures.
\newblock {\em Ann. Inst. Henri Poincar\'{e} Probab. Stat.}, 49(1):1--12, 2013.

\bibitem[CFP19]{Courtadeetal:19}
T.~A. Courtade, M.~Fathi, and A.~Pananjady.
\newblock Existence of {S}tein kernels under a spectral gap, and discrepancy
  bounds.
\newblock {\em Ann. Inst. Henri Poincar\'{e} Probab. Stat.}, 55(2):777--790,
  2019.

\bibitem[Cha07]{chatterjee:stein}
S.~Chatterjee.
\newblock Stein's method for concentration inequalities.
\newblock {\em Probab. Theory Related Fields}, 138(1-2):305--321, 2007.

\bibitem[ENT18]{Eskenazis}
A.~Eskenazis, P.~Nayar, and T.~Tkocz.
\newblock Gaussian mixtures: entropy and geometric inequalities.
\newblock {\em Ann. Probab.}, 46(5):2908--2945, 2018.

\bibitem[Fat19]{Fathi:stein}
M.~Fathi.
\newblock Stein kernels and moment maps.
\newblock {\em Ann. Probab.}, 47(4):2172--2185, 2019.

\bibitem[FKG71]{FKG}
C.~M. Fortuin, P.~W. Kasteleyn, and J.~Ginibre.
\newblock Correlation inequalities on some partially ordered sets.
\newblock {\em Comm. Math. Phys.}, 22:89--103, 1971.

\bibitem[Har04]{harge:04}
G.~Harg{\'e}.
\newblock A convex/log-concave correlation inequality for {G}aussian measure
  and an application to abstract {W}iener spaces.
\newblock {\em Probab. Theory Related Fields}, 130(3):415--440, 2004.

\bibitem[Har08]{harge}
G.~Harg\'{e}.
\newblock Reinforcement of an inequality due to {B}rascamp and {L}ieb.
\newblock {\em J. Funct. Anal.}, 254(2):267--300, 2008.

\bibitem[HP02]{HouPriv:02}
C.~Houdr\'{e} and N.~Privault.
\newblock Concentration and deviation inequalities in infinite dimensions via
  covariance representations.
\newblock {\em Bernoulli}, 8(6):697--720, 2002.

\bibitem[Hu97]{hu-chaos}
Y.~Hu.
\newblock It\^{o}-{W}iener chaos expansion with exact residual and correlation,
  variance inequalities.
\newblock {\em J. Theoret. Probab.}, 10(4):835--848, 1997.

\bibitem[Kar68]{karlin:book}
S.~Karlin.
\newblock {\em Total positivity. {V}ol. {I}}.
\newblock Stanford University Press, Stanford, Calif, 1968.

\bibitem[LaM17]{latala-matlak:royen}
R.~Lata\l~a and D.~Matlak.
\newblock Royen's proof of the {G}aussian correlation inequality.
\newblock In {\em Geometric aspects of functional analysis}, volume 2169 of
  {\em Lecture Notes in Math.}, pages 265--275. Springer, Cham, 2017.

\bibitem[Led01a]{MR1849347}
M.~Ledoux.
\newblock {\em The concentration of measure phenomenon}, volume~89 of {\em
  Mathematical Surveys and Monographs}.
\newblock American Mathematical Society, Providence, RI, 2001.

\bibitem[Led01b]{ledoux-spin}
M.~Ledoux.
\newblock Logarithmic {S}obolev inequalities for unbounded spin systems
  revisited.
\newblock In {\em S\'{e}minaire de {P}robabilit\'{e}s, {XXXV}}, volume 1755 of
  {\em Lecture Notes in Math.}, pages 167--194. Springer, Berlin, 2001.

\bibitem[LNP15]{ledoux-nourdin-peccati}
M.~Ledoux, I.~Nourdin, and G.~Peccati.
\newblock Stein's method, logarithmic {S}obolev and transport inequalities.
\newblock {\em Geom. Funct. Anal.}, 25(1):256--306, 2015.

\bibitem[LT11]{LedTal:11}
M.~Ledoux and M.~Talagrand.
\newblock {\em Probability in {B}anach spaces}.
\newblock Classics in Mathematics. Springer-Verlag, Berlin, 2011.
\newblock Isoperimetry and processes, Reprint of the 1991 edition.

\bibitem[MO13]{otto-menz}
G.~Menz and F.~Otto.
\newblock Uniform logarithmic {S}obolev inequalities for conservative spin
  systems with super-quadratic single-site potential.
\newblock {\em Ann. Probab.}, 41(3B):2182--2224, 2013.

\bibitem[NV09]{nourdin-viens}
I.~Nourdin and F.~G. Viens.
\newblock Density formula and concentration inequalities with {M}alliavin
  calculus.
\newblock {\em Electron. J. Probab.}, 14:no. 78, 2287--2309, 2009.

\bibitem[Pin10]{pinkus:book}
A.~Pinkus.
\newblock {\em Totally positive matrices}, volume 181 of {\em Cambridge Tracts
  in Mathematics}.
\newblock Cambridge University Press, Cambridge, 2010.

\bibitem[Pit82]{pitt:82}
L.~D. Pitt.
\newblock Positively correlated normal variables are associated.
\newblock {\em Ann. Probab.}, 10(2):496--499, 1982.

\bibitem[Roy14]{Royen}
T.~Royen.
\newblock A simple proof of the {G}aussian correlation conjecture extended to
  some multivariate gamma distributions.
\newblock {\em Far East J. Theor. Stat.}, 48(2):139--145, 2014.

\bibitem[Sau19]{saumard:wpi}
A.~Saumard.
\newblock Weighted {P}oincar\'{e} inequalities, concentration inequalities and
  tail bounds related to {S}tein kernels in dimension one.
\newblock {\em Bernoulli}, 25(4B):3978--4006, 2019.

\bibitem[SSZ98]{Schechtman}
G.~Schechtman, Th. Schlumprecht, and J.~Zinn.
\newblock On the {G}aussian measure of the intersection.
\newblock {\em Ann. Probab.}, 26(1):346--357, 1998.

\bibitem[SW18]{saumwellner2017efron}
A.~Saumard and J.~A. Wellner.
\newblock Efron's monotonicity property for measures on $\mathbb{R}^2$.
\newblock {\em J. {M}ultivariate {A}nal.}, 166{C}:212--224, 2018.

\bibitem[Ton90]{tong:book}
Y.~L. Tong.
\newblock {\em The multivariate normal distribution}.
\newblock Springer Series in Statistics. Springer-Verlag, New York, 1990.

\end{thebibliography}


\end{document}